\documentclass{amsart}

\usepackage{amssymb, amsmath}
\usepackage[fleqn,tbtags]{mathtools}
\mathtoolsset{showonlyrefs,showmanualtags}
\usepackage{amscd}
\usepackage[active]{srcltx}
\usepackage{verbatim}
\usepackage{epsfig} 
\usepackage{epstopdf}
\usepackage{subfigure}

\usepackage{newsymbol}
\let\emptyset \undefined
\let\ge       \undefined
\let\le       \undefined
\newsymbol\le          1336  \let\leq\le
\newsymbol\ge          133E  \let\geq\ge
\newsymbol\emptyset    203F
\newsymbol\notle       230A
\newsymbol\notge       230B

\theoremstyle{plain}
\newtheorem{theorem}{Theorem}[section]
\theoremstyle{remark}
\newtheorem{remark}[theorem]{Remark}
\newtheorem{example}[theorem]{Example}

\theoremstyle{plain}
\newtheorem{corollary}[theorem]{Corollary}
\newtheorem{lemma}[theorem]{Lemma}
\newtheorem{proposition}[theorem]{Proposition}

\numberwithin{equation}{section}

\begin{document}

\title[Parameter estimation for stochastic wave equation]
{Parameter estimation for stochastic wave equation based on observation window}

\author{Josef Jan\' ak}
\address{Department of Mathematics\\
University of Economics in Prague\\
Ekonomick\' a 957, 148 00 Prague 4\\
Czech Republic}
\email{janj04@vse.cz}



\date\today

\begin{abstract} 
Statistical inference for a linear stochastic hyperbolic equation with two unknown parameters is studied. Based on observation of coordinates of the solution or their linear combination, minimum contrast estimators are introduced. Strong consistency and asymptotic normality is proved. The results are applied to stochastic wave equation perturbed by Brownian noise and they are illustrated by a numerical simulation.
\end{abstract}

\thanks{This paper has been produced with contribution of long term
institutional support of research activities by Faculty of Informatics
and Statistics, University of Economics, Prague.\\
This paper was supported by the GA\v CR Grant no. 15--08819S}

\maketitle

\section{Introduction}
Statistical inference for stochastic partial differential equations driven by standard Brownian motion has been recently extensively studied. This paper presents results, which are interesting mostly for two reasons. In the first place, many authors use maximum likelihood estimator (MLE) for the estimation of unknown parameters in stochastic partial differential equations (SPDEs) (for example \cite{tudor}), however we are interested in minimum contrast estimator (MCE). This type of estimator has been studied since 1980's (see the "pioneering" papers \cite{koskiloges} and \cite{koskiloges 2}), but there are also more recent works. For example in \cite{maslowskitudor} and \cite{maslowskipospisil}, the (MCE) is studied even for the SPDEs driven by fractional Brownian motion.

Secondly, many authors concentrate on stochastic parabolic equations (see \cite{huebner}), but stochastic hyperbolic equations were not paid too much attention to. We may mention \cite{liulototsky} or \cite{liulototsky2}, but, again, only the (MLE) is investigated. Therefore the topic of (MCE) for stochastic hyperbolic equations (such as wave equation or plate equation) is rather new, even if the driving process is "only" standard Brownian motion.

In this work, we study parameter estimation for SPDEs of second order in time, in particular, for the following wave equation with strong damping
\begin{align}
\frac{\partial^2 u}{\partial t^2} (t, \xi) &= b\Delta u(t, \xi) - 2a \frac{\partial u}{\partial t}(t, \xi) + \eta (t, \xi), \quad (t, \xi) \in \mathbb R_+ \times D, \label{example 0} \\
u(0, \xi) &= u_1(\xi), \quad \xi \in D, \notag \\
\frac{\partial u}{\partial t} (0, \xi) &= u_2(\xi), \quad \xi \in D, \notag \\
u(t, \xi) &= 0, \quad (t, \xi) \in \mathbb R_+ \times \partial D, \notag
\end{align}
where $D \subset \mathbb R^d$ is a bounded domain with a smooth boundary and $\eta$ is a random noise.

The aim of the paper is to provide strongly consistent estimators of unknown parameters $a$ and $b$, based on the observation of the trajectory of the solution to \eqref{example 0} up to time $T$. Nevertheless, unlike our earlier work \cite{janak}, where the estimators were dependent on the observation of the norm of the solution, the present estimators will depend only on knowledge of some modes of the solution or their linear combinations (i.e., "observation window"). In order to do so, we follow up the work \cite{janak}, where the form of strongly continuous semigroup $(S(t), t \geq 0)$ generated by the operator in the drift part was found and the form of the covariance operator $Q_{\infty}^{(a,b)}$, the covariance operator of the invariant measure of the system \eqref{example 0}, was computed. Then we will use the ergodicity of the solution and an appropriate ergodic theorem.

The paper is organized as follows. The Section \ref{preliminaries} summarizes some basic facts on stochastic linear partial differential equations (which is mostly due to \cite{dapratozabczyk}) as well as setup and assumptions on the model, together with the form of the covariance operator $Q_{\infty}^{(a,b)}$ (which is acquired from \cite{janak}). In Section \ref{parameter estimation}, the family of strongly consistent estimators $(\bar{a}_T, \bar{b}_T)$ is derived. It is a specification of the general result from \cite{maslowskipospisil} to the present (hyperbolic) case. Moreover, using the "observation windows" of some special forms, even more estimators may be obtained and we present them with certain estimation strategies.

The asymptotic normality of the proposed estimators is proved in Section \ref{section: asymptotic normality}. In the end of this section, we show how these general results may be simplified for even more concrete "observation coordinates". In Section \ref{section: examples} we consider the basic example of wave equation, where our general results may be applied. These results are illustrated by some numerical simulations in Section \ref{section: implementation}.

Let us also introduce some notation. If $U$ and $V$ are Hilbert spaces then $\mathcal L(U, V)$, $\mathcal L_2 (U, V)$ and $\mathcal L_1 (U, V)$ denote the respective spaces of all linear bounded, Hilbert--Schmidt and trace class operators mapping $U$ to $V$. Also, $\mathcal L(V)$ stands for $\mathcal L(V, V)$, etc.

\section{Preliminaries} \label{preliminaries}
Given separable Hilbert spaces $U$ and $V$, we consider the equation
\begin{align}
dX(t) &= \mathcal A X(t) \, dt + \Phi \, dB(t), \label{linearequation} \\
X(0) &= x_0, \notag
\end{align}
where $(B(t), t \geq 0)$ is a standard cylindrical Brownian motion on $U$, $\mathcal A: \text{Dom}(\mathcal A) \rightarrow V, \, \text{Dom}(\mathcal A) \subset V$, $\mathcal A$ is the infinitesimal generator of a strongly continuous semigroup $(S(t), t \geq 0)$ on $V$, $\Phi \in \mathcal L (U,V)$ and $x_0 \in V$ is a random variable. We assume that $\mathbb E \| x_0 \|_V^2 < \infty$ and that $x_0$ and $(B(t), t \geq 0)$ are stochastically independent.

We impose the following two conditions:
\begin{itemize}
\item[(A1)] $\Phi \in \mathcal L_2 (U,V)$,
\item[(A2)] There exist constants $K > 0$ and $\rho > 0$ such that
\begin{equation}
\| S(t) \|_{\mathcal L (V)} \leq Ke^{- \rho t}
\end{equation}
holds for all $t \geq 0$.
\end{itemize}

The condition (A1) means that the perturbing noise is, in fact, a genuine $V$--valued Brownian motion and the condition (A2) is the exponential stability of the semigroup generated by $\mathcal A$.

The following two Propositions describe the form of a mild solution to the equation \eqref{linearequation} and its invariant measure (cf. \cite{dapratozabczyk}).

\begin{proposition}
If (A1) is satisfied then equation \eqref{linearequation} admits a mild solution
\begin{equation}
X^{x_0}(t) = S(t)x_0 + Z(t), \quad t \geq 0,
\end{equation}
where $(Z(t), t \geq 0)$ is the convolution integral
\begin{equation}
Z(t) = \int_0^t S(t-u) \Phi \, dB(u), \quad t \geq 0.
\end{equation}

The process $(Z(t), t \geq 0)$ is a $V$--continuous centered Gaussian process with covariance operator given by the formula
\begin{equation}
Q_t = \int_0^t S(u) \Phi \Phi^* S^* (u) \, du.
\end{equation}
\end{proposition}

\begin{proposition} \label{existence of invariant measure}
If (A1), (A2) are satisfied then there is a unique invariant measure $\mu_{\infty} = N \left( 0, Q_{\infty} \right)$ for the equation \eqref{linearequation} and
$$
w^* - \lim_{t \rightarrow \infty} \mu_t^{x_0} = \mu_{\infty}
$$
for each initial condition $x_0 \in V$, where $\mu_t^{x_0} = \mathrm{Law} \, (X^{x_0}(t))$ and $\mathrm{Law} \, (\cdot)$ denotes the probability distribution.

The covariance operator $Q_{\infty}$ takes the form
\begin{equation} \label{qinfty}
Q_{\infty} = \int_0^{\infty} S(t) \Phi \Phi^* S^* (t) \, dt.
\end{equation}
\end{proposition}

To interpret stochastic wave equation \eqref{example 0} rigorously, we rewrite it as a first order system in a standard way. Assume that $\{e_n, n \in \mathbb N \}$ is an orthonormal basis in $L^2(D)$ and the operator $A: \mathrm{Dom}(A) \subset L^2(D) \rightarrow L^2(D)$ is such that
\begin{itemize}
\item[(A3)] $Ae_n = - \alpha_n e_n$, \quad $\forall n \in \mathbb N \quad \alpha_n > 0$, \quad $\alpha_n \rightarrow \infty$ for $n \rightarrow \infty$.
\end{itemize}

These assumptions cover the case when the set $D \subset \mathbb R^d$ is open, bounded and the boundary $\partial D$ is sufficiently smooth, the operator $A = \Delta|_{\mathrm{Dom}(A)}$ and $\mathrm{Dom}(A) = H^2(D) \cap H^1_0 (D)$.

Next let us assume that $\Phi_1$ is a Hilbert--Schmidt operator on $L^2(D)$ such that $Q = \Phi_1 \Phi_1^*$ satisfies
\begin{itemize}
\item[(A4)] $Qe_n = \lambda_n e_n$, \quad $\forall n \in \mathbb N \quad \lambda_n > 0$, \quad $\sum_{n=1}^{\infty} \lambda_n < \infty$.
\end{itemize}

The assumption (A4) means that we consider so--called "diagonal case", i.e., operators $A$ and $Q$ have common set of eigenvectors $\{e_n, n \in \mathbb N \}$.

Consider the Hilbert space $V = \mathrm{Dom}((-A)^{\frac{1}{2}}) \times L^2(D)$ endowed with the inner product
\begin{align}
\left\langle \left( \begin{array}{r}
x_1\\
x_2
\end{array} \right), \left( \begin{array}{r}
y_1\\
y_2
\end{array} \right) \right\rangle_V &= \left\langle x_1, y_1 \right\rangle_{\mathrm{Dom}((-A)^{\frac{1}{2}})} + \left\langle x_2, y_2 \right\rangle_{L^2(D)} \notag \\
&= \left\langle (-A)^{\frac{1}{2}}x_1, (-A)^{\frac{1}{2}}y_1 \right\rangle_{L^2(D)} + \left\langle x_2, y_2 \right\rangle_{L^2(D)},
\end{align}
for any $(x_1, x_2)^{\top}, (y_1, y_2)^{\top} \in V$.

Also, consider the linear equation
\begin{align}
dX(t) &= \mathcal A X(t) \, dt + \Phi \, dB(t), \label{linear equation with parameters} \\
X(0) &= x_0 = \left( \begin{array}{c}
u_1\\
u_2
\end{array} \right), \notag
\end{align}
where the linear operator $\mathcal A: \mathrm{Dom}(\mathcal A) = \mathrm{Dom}(A) \times \mathrm{Dom}((-A)^{\frac{1}{2}}) \rightarrow V$ is defined by
$$
\mathcal Ax =  \mathcal A \left( \begin{array}{r}
x_1\\
x_2
\end{array} \right) = \left( \begin{array}{cc}
0&I\\
bA&-2aI
\end{array} \right) \left( \begin{array}{r}
x_1\\
x_2
\end{array} \right), \quad \forall x = \left( \begin{array}{r}
x_1\\
x_2
\end{array} \right) \in \mathrm{Dom}(\mathcal A),
$$
$a > 0, \, b > 0$ are unknown parameters (which are to be estimated), $u_1 \in \mathrm{Dom}((-A)^{\frac{1}{2}}),$ $u_2 \in L^2(D)$, $x_0 = (u_1, u_2)^{\top} \in V$ satisfies $\mathbb E \| x_0 \|_V^2 < \infty$, where $\| \cdot \|_V := \sqrt{\left\langle \cdot, \cdot \right\rangle_V}$, and the linear operator $\Phi : U = V \rightarrow V$ is defined by
$$
\Phi = \left( \begin{array}{cc}
0&0\\
0&\Phi_1
\end{array} \right).
$$

With no loss of generality, we assume that the driving process in \eqref{linear equation with parameters} takes the form $(0, B(t))^{\top}$, where $(B(t), t \geq 0)$ is a standard cylindrical Brownian motion on $L^2(D)$.

Note that since the operator $\Phi_1$ is Hilbert--Schmidt in $L^2(D)$, the operator $\Phi$ is Hilbert--Schmidt in $V$. Also note that the orthonormal basis of the space $\mathrm{Dom}((-A)^{\frac{1}{2}})$ is $\{ f_n, n \in \mathbb N \}$, where $f_n = \frac{1}{\sqrt{\alpha_n}} e_n$.

The operator $\mathcal A$ is the infinitesimal generator of the strongly continuous semigroup $(S(t), t \geq 0)$ on $V$, which is also exponentially stable (see \cite{janak}, Theorem 3.10). Its exact form is not needed in the sequel, however we will need the formula for the covariance operator $Q_{\infty}^{(a,b)}$.

\begin{theorem} \label{form of Q infinity - theorem}
If (A1) -- (A4) are satisfied then there is a unique invariant measure $\mu_{\infty}^{(a,b)} = N \left( 0, Q_{\infty}^{(a,b)} \right)$ for the equation \eqref{linear equation with parameters} and
$$
w^* - \lim_{t \rightarrow \infty} \mu_t^{x_0} = \mu_{\infty}^{(a,b)}
$$
for each initial condition $x_0 \in V$. The covariance operator $Q_{\infty}^{(a,b)}$ takes the form
\begin{align}
Q_{\infty}^{(a,b)} \left( \begin{array}{c}
x_1\\
x_2
\end{array} \right) &= \sum_{n=1}^{\infty} \sum_{k=1}^{\infty} \frac{\left\langle Qe_n, e_k \right\rangle_{L^2(D)}}{b^2 (\alpha_n - \alpha_k)^2 + 8a^2 b (\alpha_n + \alpha_k)} \times \\
&\left( \begin{array}{c}
4a \alpha_n \left\langle x_1, e_n \right\rangle_{L^2(D)} e_k + b(\alpha_k - \alpha_n) \left\langle x_2, e_n \right\rangle_{L^2(D)} e_k\\
b \alpha_n(\alpha_n - \alpha_k) \left\langle x_1, e_n \right\rangle_{L^2(D)} e_k + 2ab (\alpha_n + \alpha_k) \left\langle x_2, e_n \right\rangle_{L^2(D)} e_k
\end{array} \right), \label{form of Q infinity general case}
\end{align}
for any $(x_1, x_2)^{\top} \in V.$
\end{theorem}
\begin{proof}
See \cite{janak}, Theorem 3.11.
\end{proof}

Since we consider only the diagonal case, we will be working with the simplified version of the above formula, that is
\begin{equation} \label{form of Q infinity diagonal case}
Q_{\infty}^{(a,b)} = \left( \begin{array}{cc} \frac{1}{4ab} \, Q & 0 \\ 0 & \frac{1}{4a} \, Q \end{array} \right).
\end{equation}

\section{Parameter estimation} \label{parameter estimation}
Consider the stochastic differential equation \eqref{linear equation with parameters} with the parameters $a > 0$, $b > 0$ unknown. Our aim is to propose strongly consistent estimators of these parameters based on observation of the trajectory through some "observation window" $z \in V$.

More specifically, let $ 0 \neq z \in V$ be arbitrary and consider that we are able to track the trajectory of the process $\left( \left\langle X^{x_0}(t), z \right\rangle_V, 0 \leq t \leq T \right)$ up to time $T$.

Since the linear differential equation \eqref{linear equation with parameters} has unique invariant measure $\mu_{\infty}^{(a,b)}$, we may use the following ergodic theorem for arbitrary solution (see \cite{maslowskipospisil}, Theorem 4.9).

\begin{theorem} \label{ergodic theorem}
Let $(X^{x_0}(t), t \geq 0)$ be a solution to \eqref{linear equation with parameters} with $\Phi \in \mathcal L_2(U,V)$. Let $\varrho : V \rightarrow \mathbb R$ be a functional satisfying the following local Lipschitz condition: let there exist real constants $K > 0$ and $m \geq 0$ such that
\begin{equation}
|\varrho (x) - \varrho (y)| \leq K \| x - y \|_V \left( 1 + \| x \|_V^m + \| y \|_V^m \right)
\end{equation}
for all $x, y \in V$. Then
\begin{equation}
\lim_{T \rightarrow \infty} \frac{1}{T} \int_0^T \varrho \left( X^{x_0}(t) \right) \, dt = \int_V \varrho (y) \, \mu_{\infty} (dy), \quad \mathbb P-a.s.
\end{equation}
for all $x_0 \in V$.
\end{theorem}

Let $z \in V$ be arbitrary. Using a functional $\varrho : V \rightarrow \mathbb R$, $\varrho (y) = \left\langle y,z \right\rangle_V^2$, $y \in V$, all the conditions of Theorem \ref{ergodic theorem} will be satisfied with $m=1$ and
\begin{align}
\lim_{T \rightarrow \infty} \frac{1}{T} \int_0^T \varrho \left( X^{x_0}(t) \right) \, dt &= \lim_{T \rightarrow \infty} \frac{1}{T} \int_0^T \left\langle X^{x_0}(t), z \right\rangle_V^2 \, dt \notag \\
&= \int_V \left\langle y,z \right\rangle_V^2 \, \mu_{\infty}^{(a,b)} (dy) \notag \\
&= \left\langle Q_{\infty}^{(a,b)} z, z \right\rangle_V, \quad \mathbb P-a.s. \label{ergodic property-window}
\end{align}

Based on above convergence, some strongly consistent estimators of parameters $a$ and $b$ may be proposed.

\begin{theorem} \label{abarbbar - general case}
Let $z \in V$ be arbitrary, $z = (z_1, z_2)^{\top}$, where $z_1 \in \mathrm{Dom}((-A)^{\frac{1}{2}})$, $z_2 \in L^2(D)$. Define
\begin{equation}
J_T = \frac{1}{T} \int_0^T \left\langle X^{x_0}(t), z \right\rangle_V^2 \, dt.
\end{equation}

1) If $z \neq 0$ then the process
\begin{equation} \label{abar - general case}
\bar{a}_T = \frac{1}{4b J_T} \left\langle Q z_1, z_1\right\rangle_{\mathrm{Dom}((-A)^{\frac{1}{2}})} + \frac{1}{4 J_T} \left\langle Q z_2, z_2 \right\rangle_{L^2(D)}
\end{equation}
is strongly consistent estimator of the parameter $a$, i.e., $\bar{a}_T \rightarrow a$, $\mathbb P-a.s.$ as $T \rightarrow \infty$.

2) If $z_1 \neq 0$ then the process
\begin{equation} \label{bbar - general case}
\bar{b}_T = \frac{\left\langle Q z_1, z_1\right\rangle_{\mathrm{Dom}((-A)^{\frac{1}{2}})}}{4a J_T - \left\langle Q z_2, z_2 \right\rangle_{L^2(D)}}
\end{equation}
is strongly consistent estimator of the parameter $b$, i.e., $\bar{b}_T \rightarrow b$, $\mathbb P-a.s.$ as $T \rightarrow \infty$.
\end{theorem}
\begin{proof}
From \eqref{ergodic property-window} and \eqref{form of Q infinity diagonal case} it follows that
\begin{align}
\lim_{T \rightarrow \infty} J_T &= \left\langle Q_{\infty}^{(a,b)} z, z \right\rangle_V = \left\langle \left( \begin{array}{cc} \frac{1}{4ab} \, Q & 0 \\ 0 & \frac{1}{4a} \, Q \end{array} \right) \left( \begin{array}{c} z_1 \\ z_2 \end{array} \right), \left( \begin{array}{c} z_1 \\ z_2 \end{array} \right) \right\rangle_V \notag \\
&= \frac{1}{4ab} \left\langle Q z_1, z_1\right\rangle_{\mathrm{Dom}((-A)^{\frac{1}{2}})} + \frac{1}{4a} \left\langle Q z_2, z_2 \right\rangle_{L^2(D)}, \quad \mathbb P-a.s. \notag
\end{align}
Hence we obtain the desired convergence $\bar{a}_T \rightarrow a$, $\mathbb P-a.s.$ as $T \rightarrow \infty$ unless $z = 0$. Similarly, if $z_1 \neq 0$ then we obtain the convergence $\bar{b}_T \rightarrow b$, $\mathbb P-a.s.$ as $T \rightarrow \infty$.
\end{proof}

The estimators $\bar{a}_T$ and $\bar{b}_T$ have one major disadvantage: In order to compute the estimator $\bar{a}_T$, we have to know the true value of the other parameter $b$ (and vice versa for the estimator $\bar{b}_T$). This problem may be overcome by a more specific choice of the "observation window". Therefore, consider the following special cases or estimation strategies: \\

1. If $z = (0, z_2)^{\top}$, i.e., $z_1 = 0$, $z_2 \neq 0$, then
\begin{equation}
\bar{a}_T = \frac{\left\langle Q z_2, z_2 \right\rangle_{L^2(D)}}{\frac{4}{T} \int_0^T \left\langle X_2^{x_0}(t), z_2 \right\rangle_{L^2(D)}^2 \, dt},
\end{equation}
where $X^{x_0}(t) = (X_1^{x_0}(t), X_2^{x_0}(t))^{\top} \in V$ is the solution to the equation \eqref{linear equation with parameters}. In order to make such an estimator, only the observation of the second component of the solution is needed. \\

2. If $z = (z_1, 0)^{\top}$, i.e., $z_2 = 0$, $z_1 \neq 0$, then we have that
\begin{equation}
\lim_{T \rightarrow \infty} J_T = \lim_{T \rightarrow \infty} \frac{1}{T} \int_0^T \left\langle X_1^{x_0}(t), z_1 \right\rangle_{\mathrm{Dom}((-A)^{\frac{1}{2}})}^2 \, dt = \frac{1}{4ab} \left\langle Q z_1, z_1\right\rangle_{\mathrm{Dom}((-A)^{\frac{1}{2}})}
\end{equation}
and it is possible to estimate either the product $ab$, or one of the parameters if the true value of the other one is known. (In this case the formulae \eqref{abar - general case} and \eqref{bbar - general case} actually coincide.) \\

3. It is possible to combine the two previous strategies together. First, using the "window" $z = (0, z_2)^{\top}$, $z_2 \neq 0$, we get an estimator of $a$, that is
\begin{equation}
\bar{a}_T = \frac{\left\langle Q z_2, z_2 \right\rangle_{L^2(D)}}{\frac{4}{T} \int_0^T \left\langle X_2^{x_0}(t), z_2 \right\rangle_{L^2(D)}^2 \, dt}
\end{equation}
and then, using the "window" $z = (z_1, 0)^{\top}$, $z_1 \neq 0$, we get an estimator of $b$, that is
\begin{equation}
\bar{b}_T = \frac{\left\langle Q z_1, z_1\right\rangle_{\mathrm{Dom}((-A)^{\frac{1}{2}})}}{\left\langle Q z_2, z_2 \right\rangle_{L^2(D)}} \cdot \frac{\frac{1}{T} \int_0^T \left\langle X_2^{x_0}(t), z_2 \right\rangle_{L^2(D)}^2 \, dt}{\frac{1}{T} \int_0^T \left\langle X_1^{x_0}(t), z_1 \right\rangle_{\mathrm{Dom}((-A)^{\frac{1}{2}})}^2 \, dt}.
\end{equation} \\

4. It is also possible to generalize the previous procedure further. First, using the "window" $z = (0, z_2)^{\top}$, $z_2 \neq 0$, we get an estimator of $a$, that is
\begin{equation}
\bar{a}_T = \frac{\left\langle Q z_2, z_2 \right\rangle_{L^2(D)}}{\frac{4}{T} \int_0^T \left\langle X_2^{x_0}(t), z_2 \right\rangle_{L^2(D)}^2 \, dt}
\end{equation}
and then, using any "window" $\tilde{z} = (\tilde{z}_1, \tilde{z}_2)^{\top}$, $\tilde{z}_1 \neq 0$,  we get an estimator of $b$, that is
\begin{align}
\bar{b}_T &= \frac{\left\langle Q \tilde{z}_1, \tilde{z}_1\right\rangle_{\mathrm{Dom}((-A)^{\frac{1}{2}})}}{4 \bar{a}_T \frac{1}{T} \int_0^T \left\langle X^{x_0}(t), \tilde{z} \right\rangle_V^2 \, dt - \left\langle Q \tilde{z}_2, \tilde{z}_2 \right\rangle_{L^2(D)}} \notag \\
&= \frac{\left\langle Q \tilde{z}_1, \tilde{z}_1\right\rangle_{\mathrm{Dom}((-A)^{\frac{1}{2}})}}{\left\langle Q z_2, z_2 \right\rangle_{L^2(D)} \frac{\frac{1}{T} \int_0^T \left\langle X^{x_0}(t), \tilde{z} \right\rangle_V^2 \, dt}{\frac{1}{T} \int_0^T \left\langle X_2^{x_0}(t), z_2 \right\rangle_{L^2(D)}^2 \, dt} - \left\langle Q \tilde{z}_2, \tilde{z}_2 \right\rangle_{L^2(D)}} \notag \\
&= \frac{\left\langle Q \tilde{z}_1, \tilde{z}_1\right\rangle_{\mathrm{Dom}((-A)^{\frac{1}{2}})} \frac{1}{T} \int_0^T \left\langle X_2^{x_0}(t), z_2 \right\rangle_{L^2(D)}^2 \, dt}{\left\langle Q z_2, z_2 \right\rangle_{L^2(D)} \frac{1}{T} \int_0^T \left\langle X^{x_0}(t), \tilde{z} \right\rangle_V^2 \, dt - \left\langle Q \tilde{z}_2, \tilde{z}_2 \right\rangle_{L^2(D)} \frac{1}{T} \int_0^T \left\langle X_2^{x_0}(t), z_2 \right\rangle_{L^2(D)}^2 \, dt}. \notag
\end{align} \\

For the practical reasons there is an incentive to observe the solution to the equation \eqref{linear equation with parameters} through the "observation window" componentwise (i.e., we observe the processes $(\left\langle X_1^{x_0}, z_1 \right\rangle_{\mathrm{Dom}((-A)^{\frac{1}{2}})}, 0 \leq t \leq T)$ and $(\left\langle X_2^{x_0}, z_2 \right\rangle_{L^2(D)}, 0 \leq t \leq T)$ for any given $0 \neq z_1 \in V$, $0 \neq z_2 \in \mathrm{Dom}((-A)^{\frac{1}{2}})$ separately), so we will prefer the strategy 3. Let us introduce these estimators once again with the new notation.

\begin{corollary}
1) Let $0 \neq z_2 \in L^2(D)$ be arbitrary. The process
\begin{equation} \label{abar}
\bar{a}_{T, z_2} = \frac{\left\langle Q z_2, z_2 \right\rangle_{L^2(D)}}{\frac{4}{T} \int_0^T \left\langle X_2^{x_0}(t), z_2 \right\rangle_{L^2(D)}^2 \, dt}
\end{equation}
is strongly consistent estimator of the parameter $a$.

2) Moreover, let $0 \neq z_1 \in \mathrm{Dom}((-A)^{\frac{1}{2}})$ be arbitrary. The process
\begin{equation} \label{bbar}
\bar{b}_{T, {z_1, z_2}} = \frac{\left\langle Q z_1, z_1\right\rangle_{\mathrm{Dom}((-A)^{\frac{1}{2}})}}{\left\langle Q z_2, z_2 \right\rangle_{L^2(D)}} \cdot \frac{\frac{1}{T} \int_0^T \left\langle X_2^{x_0}(t), z_2 \right\rangle_{L^2(D)}^2 \, dt}{\frac{1}{T} \int_0^T \left\langle X_1^{x_0}(t), z_1 \right\rangle_{\mathrm{Dom}((-A)^{\frac{1}{2}})}^2 \, dt}
\end{equation}
is strongly consistent estimator of the parameter $b$.
\end{corollary}
\begin{proof}
It is the direct consequence of Theorem \ref{abarbbar - general case}.
\end{proof}

The previous estimators may be specified even further if the "observation window" is the element of the orthonormal basis.

Indeed, if $z_2 = e_k \in L^2(D)$ for any $k \in \mathbb N$ then $\bar{a}_{T, z_2}$ takes the form
\begin{equation} \label{abark}
\bar{a}_{T, k} := \bar{a}_{T, e_k} = \frac{\lambda_k}{\frac{4}{T} \int_0^T \left\langle X_2^{x_0}(t), e_k \right\rangle_{L^2(D)}^2 \, dt}.
\end{equation}
Moreover, if $z_1 = f_j \in \mathrm{Dom}((-A)^{\frac{1}{2}})$ for any $j \in \mathbb N$ then $\bar{b}_{T, z_1, z_2}$ takes the form
\begin{equation} \label{bbarjk}
\bar{b}_{T, j, k} := \bar{b}_{T, {f_j, e_k}} = \frac{\lambda_j}{\lambda_k} \cdot \frac{\frac{1}{T} \int_0^T \left\langle X_2^{x_0}(t), e_k \right\rangle_{L^2(D)}^2 \, dt}{\frac{1}{T} \left\langle X_1^{x_0}(t), f_j \right\rangle_{\mathrm{Dom}((-A)^{\frac{1}{2}})}^2 \, dt}.
\end{equation}

These estimators are using only observation of some given modes in the expansion of the solution.

\section{Asymptotic normality of the estimators} \label{section: asymptotic normality}
\subsection{Asymptotic normality of the estimator $\bar{a}_{T, z_2}$}
In this section we show asymp\-to\-tic normality of the estimator \eqref{abar}, i.e., the weak convergence of $\sqrt{T} (\bar{a}_{T, z_2} - a)$ to a Gaussian distribution.

Let $k \in \mathbb N$ be arbitrary and define the operator $E_k : V \rightarrow V$ by
\begin{equation}
E_k x = E_k \left( \begin{array}{c} x_1 \\ x_2 \end{array} \right) = \left( \begin{array}{cc} E_{k,1} & E_{k,2} \\ E_{k,3} & E_{k,4} \end{array} \right) \left( \begin{array}{c} x_1 \\ x_2 \end{array} \right), \quad \forall x = \left( \begin{array}{c} x_1 \\ x_2 \end{array} \right) \in V,
\end{equation}
where $E_{k,2} = 0$, $E_{k,3} = 0$ and
\begin{align}
E_{k,1} : x_1 &\longmapsto b \left\langle x_1, e_k \right\rangle_{L^2(D)} e_k, \notag \\
E_{k,4} : x_2 &\longmapsto \left\langle x_2, e_k \right\rangle_{L^2(D)} e_k, \notag
\end{align}
for any $x_1 \in \mathrm{Dom}((-A)^{\frac{1}{2}})$ and $x_2 \in L^2(D)$. Hence the operator $E_k$ evaluates as
\begin{equation}
E_k x = \left( \begin{array}{c} b \left\langle x_1, e_k \right\rangle_{L^2(D)} e_k \\ \left\langle x_2, e_k \right\rangle_{L^2(D)} e_k \end{array} \right), \quad \forall x = \left( \begin{array}{c} x_1 \\ x_2 \end{array} \right) \in V.
\end{equation}

Let $k, l \in \mathbb N$ be arbitrary and define the operator $E_{k,l} : V \rightarrow V$ by
\begin{equation}
E_{k,l} x = E_{k,l} \left( \begin{array}{c} x_1 \\ x_2 \end{array} \right) = \frac{1}{D_{k,l}} \left( \begin{array}{cc} E_{k,l,1} & E_{k,l,2} \\ E_{k,l,3} & E_{k,l,4} \end{array} \right) \left( \begin{array}{c} x_1 \\ x_2 \end{array} \right), \quad \forall x = \left( \begin{array}{c} x_1 \\ x_2 \end{array} \right) \in V,
\end{equation}
where
\begin{align}
E_{k,l,1} : x_1 &\longmapsto 16 a^2 b \alpha_l \left\langle x_1, e_l \right\rangle_{L^2(D)} e_k + 16 a^2 b \alpha_k \left\langle x_1, e_k \right\rangle_{L^2(D)} e_l, \notag \\
E_{k,l,2} : x_2 &\longmapsto 4ab (\alpha_k - \alpha_l) \left\langle x_2, e_k \right\rangle_{L^2(D)} e_l + 4ab (\alpha_l - \alpha_k) \left\langle x_2, e_l \right\rangle_{L^2(D)} e_k, \notag \\
E_{k,l,3} : x_1 &\longmapsto 4ab \alpha_l (\alpha_k - \alpha_l) \left\langle x_1, e_l \right\rangle_{L^2(D)} e_k + 4ab \alpha_k (\alpha_l - \alpha_k) \left\langle x_1, e_k \right\rangle_{L^2(D)} e_l, \notag \\
E_{k,l,4} : x_2 &\longmapsto 8a^2 (\alpha_k + \alpha_l) \left\langle x_2, e_k \right\rangle_{L^2(D)} e_l + 8a^2 (\alpha_k + \alpha_l) \left\langle x_2, e_l \right\rangle_{L^2(D)} e_k, \notag
\end{align}
for any $x_1 \in \mathrm{Dom}((-A)^{\frac{1}{2}})$, $x_2 \in L^2(D)$ and $D_{k,l}$ defined by
\begin{equation}
D_{k,l} = b(\alpha_k - \alpha_l)^2 + 8a^2 (\alpha_k + \alpha_l).
\end{equation}

Note that $D_{k,l}$ is the denominator from the formula \eqref{form of Q infinity general case} divided by $b$ and that $D_{k,l} = D_{l,k}$. Also note that $D_{k,k} = 16 a^2 \alpha_k$.

The properties of the operators $E_k$ and $E_{k,l}$ are summarized in the following Lemma.

\begin{lemma}
1) The operator $E_k \in \mathcal L(V)$ is self--adjoint for any given $k \in \mathbb N$. Moreover,
\begin{equation} \label{EkxAx}
\left\langle E_k x, \mathcal A x \right\rangle_V = -2a \left\langle x_2, e_k \right\rangle_{L^2(D)}^2, \quad \forall x = \left( \begin{array}{c} x_1 \\ x_2 \end{array} \right) \in \mathrm{Dom}(\mathcal A).
\end{equation}

2) The operator $E_{k,l}\in \mathcal L(V)$ is self--adjoint for any given $k, l \in \mathbb N$. Moreover,
\begin{equation} \label{EklxAx}
\left\langle E_{k,l} x, \mathcal A x \right\rangle_V = -4a \left\langle x_2, e_k \right\rangle_{L^2(D)} \left\langle x_2, e_l \right\rangle_{L^2(D)}, \quad \forall x = \left( \begin{array}{c} x_1 \\ x_2 \end{array} \right) \in \mathrm{Dom}(\mathcal A).
\end{equation}
\end{lemma}
\begin{proof}
1) Let $k \in \mathbb N$ be arbitrary. It is evident that $E_k \in \mathcal L (V)$ and for $x = (x_1, x_2)^{\top} \in V$ and $y = (y_1, y_2)^{\top} \in V$ we have
\begin{align}
\left\langle E_k x, y \right\rangle_V &= \left\langle \left( \begin{array}{c} b \left\langle x_1, e_k \right\rangle_{L^2(D)} e_k \\ \left\langle x_2, e_k \right\rangle_{L^2(D)} e_k \end{array} \right), \left( \begin{array}{c} y_1 \\ y_2 \end{array} \right) \right\rangle_V \notag \\
&= b \left\langle x_1, e_k \right\rangle_{L^2(D)} \left\langle y_1, e_k \right\rangle_{\mathrm{Dom}((-A)^{\frac{1}{2}})} + \left\langle x_2, e_k \right\rangle_{L^2(D)} \left\langle y_2, e_k \right\rangle_{L^2(D)} \notag \\
&= b \alpha_k \left\langle x_1, e_k \right\rangle_{L^2(D)} \left\langle y_1, e_k \right\rangle_{L^2(D)} + \left\langle x_2, e_k \right\rangle_{L^2(D)} \left\langle y_2, e_k \right\rangle_{L^2(D)} \notag \\
&= \left\langle x, E_k y \right\rangle_V, \notag
\end{align}
hence $E_k = E_k^*$. Moreover, for every $x = (x_1, x_2)^{\top} \in \mathrm{Dom}(\mathcal A)$ we have
\begin{align}
\left\langle E_k x, \mathcal A x \right\rangle_V &= \left\langle \left( \begin{array}{c} b \left\langle x_1, e_k \right\rangle_{L^2(D)} e_k \\ \left\langle x_2, e_k \right\rangle_{L^2(D)} e_k \end{array} \right), \left( \begin{array}{c} x_2 \\ bAx_1 - 2ax_2 \end{array} \right) \right\rangle_V \notag \\
&= b \left\langle x_1, e_k \right\rangle_{L^2(D)} \left\langle (-A)^{\frac{1}{2}} x_2, (-A)^{\frac{1}{2}} e_k \right\rangle_{L^2(D)} \notag \\
&\phantom{=}+ b \left\langle x_2, e_k \right\rangle_{L^2(D)} \left\langle A x_1, e_k \right\rangle_{L^2(D)} - 2a \left\langle x_2, e_k \right\rangle_{L^2(D)}^2 \notag \\
&= -2a \left\langle x_2, e_k \right\rangle_{L^2(D)}^2. \notag
\end{align}

2) Let $k,l \in \mathbb N$ be arbitrary. It is evident that $E_{k,l} \in \mathcal L (V)$ and similarly as above it is possible to verify that $E_{k,l} = E_{k,l}^*$ and that \eqref{EklxAx} holds true for any $x \in \mathrm{Dom}(\mathcal A)$.
\end{proof}

Choose $0 \neq z_2 \in L^2(D)$ taking the form
\begin{equation}
z_2 = \sum_{k=1}^{\infty} \left\langle z_2, e_k \right\rangle_{L^2(D)} e_k = \sum_{k=1}^{\infty} z_{2,k} e_k.
\end{equation}

Finally, define the operator $E : V \rightarrow V$ by
\begin{equation} \label{definition of E}
E = \sum_{k=1}^{\infty} z_{2,k}^2 E_k + \sum_{k=1}^{\infty} \sum_{l=k+1}^{\infty} z_{2,k} z_{2,l} E_{k,l}.
\end{equation}

The properties of the operator $E$ needed in the sequel are summarized in the following Lemma.

\begin{lemma} \label{properties of E}
The operator $E \in \mathcal L(V)$. Moreover, it is self--adjoint and
\begin{equation} \label{ExAx}
\left\langle E x, \mathcal A x \right\rangle_V = -2a \left\langle x_2, z_2 \right\rangle_{L^2(D)}^2, \quad \forall x = \left( \begin{array}{c} x_1 \\ x_2 \end{array} \right) \in \mathrm{Dom}(\mathcal A).
\end{equation}
\end{lemma}
\begin{proof}
There exists a positive constant $C > 0$ (which does not depend on $k$) such that $\| E_k \|_{\mathcal L(V)} < C$ for any $k \in \mathbb N$. Hence
\begin{equation}
\left\| \sum_{k=1}^{\infty} z_{2,k}^2 E_k \right\|_{\mathcal L(V)} \leq \sum_{k=1}^{\infty} z_{2,k}^2 \| E_k \|_{\mathcal L(V)} \leq C \sum_{k=1}^{\infty} z_{2,k}^2 = C \| z_2 \|_{L^2(D)}^2 < \infty
\end{equation}
and the operator defined by the first sum in \eqref{definition of E} belongs to the space $\mathcal L(V)$.

The convergence of the double series is fulfilled by the denominator $D_{k,l}$. For any $k,l \in \mathbb N$ we have
\begin{equation}
\frac{\alpha_k + \alpha_l}{D_{k,l}} \leq \frac{\alpha_k + \alpha_l}{8a^2(\alpha_k + \alpha_l)} = \frac{1}{8a},
\end{equation}
\begin{align}
\frac{\sqrt{\alpha_k} |\alpha_k - \alpha_l|}{D_{k,l}} &= \sqrt{\frac{\alpha_k}{D_{k,l}}} \frac{|\alpha_k - \alpha_l|}{\sqrt{D_{k,l}}} < \frac{1}{\sqrt{8a^2}} \frac{|\alpha_k - \alpha_l|}{\sqrt{b(\alpha_k - \alpha_l)^2 + 8a^2 (\alpha_k + \alpha_l)}} \notag \\
&< \frac{1}{\sqrt{8a^2}} \frac{|\alpha_k - \alpha_l|}{\sqrt{b} |\alpha_k - \alpha_l|} = \frac{1}{8a^2 b}, \notag
\end{align}
\begin{equation}
\frac{\sqrt{\alpha_k \alpha_l}}{D_{k,l}} < \frac{2 \sqrt{\alpha_k \alpha_l}}{D_{k,l}} \leq \frac{2 \sqrt{\alpha_k \alpha_l}}{8a^2 (\alpha_k + \alpha_l)} \leq \frac{1}{8a^2}.
\end{equation}

The desired convergence is then accomplished by the convergence of the series
\begin{equation}
\sum_{k=1}^{\infty} z_{2,k}^2 < \infty, \quad \sum_{k=1}^{\infty} \left\langle x_2, e_k \right\rangle_{L^2(D)}^2 < \infty, \quad \sum_{k=1}^{\infty} \alpha_k \left\langle x_1, e_k \right\rangle_{L^2(D)}^2 < \infty.
\end{equation}

The linear combination of the self--adjoint operators is also the self--adjoint ope\-ra\-tor, hence $E = E^*$.

The property \eqref{ExAx} follows by \eqref{EkxAx}, \eqref{EklxAx} and the following computation
\begin{align}
&\left\langle E x, \mathcal A x \right\rangle_V = \notag \\
&= \left\langle \left( \sum_{k=1}^{\infty} z_{2,k}^2 E_k + \sum_{k=1}^{\infty} \sum_{l=k+1}^{\infty} z_{2,k} z_{2,l} E_{k,l} \right) x, \mathcal A x \right\rangle_V \notag \\
&= \sum_{k=1}^{\infty} z_{2,k}^2 \left\langle E_k x, \mathcal A x \right\rangle_V + \sum_{k=1}^{\infty} \sum_{l=k+1}^{\infty} z_{2,k} z_{2,l} \left\langle E_{k,l} x, \mathcal A x \right\rangle_V \notag \\
&= -2a \left( \sum_{k=1}^{\infty} z_{2,k}^2 \left\langle x_2, e_k \right\rangle_{L^2(D)}^2 + 2 \sum_{k=1}^{\infty} \sum_{l=k+1}^{\infty} z_{2,k} z_{2,l} \left\langle x_2, e_k \right\rangle_{L^2(D)} \left\langle x_2, e_l \right\rangle_{L^2(D)} \right) \notag \\
&= -2a \left( \sum_{k=1}^{\infty} z_{2,k} \left\langle x_2, e_k \right\rangle_{L^2(D)} \right)^2 \notag \\
&= -2a \left\langle x_2, \sum_{k=1}^{\infty} z_{2,k} e_k \right\rangle_{L^2(D)}^2 \notag \\
&= -2a \left\langle x_2, z_2 \right\rangle_{L^2(D)}^2. \notag
\end{align}
\end{proof}

We will need an alternative representation for the process \\ $\frac{1}{T} \int_0^T \left\langle X_2^{x_0}(t), z_2 \right\rangle_{L^2(D)}^2 \, dt$.

\begin{lemma} \label{ito for z2-lemma}
The process $\frac{1}{T} \int_0^T \left\langle X_2^{x_0}(t), z_2 \right\rangle_{L^2(D)}^2 \, dt$ admits the following representation
\begin{align}
\frac{1}{T} \int_0^T \left\langle X_2^{x_0}(t), z_2 \right\rangle_{L^2(D)}^2 \, dt &= - \frac{1}{4aT} \left( \left\langle E X^{x_0}(T), X^{x_0}(T) \right\rangle_V - \left\langle E x_0, x_0 \right\rangle_V \right) \notag \\
&\phantom{=} + \frac{1}{2aT} \int_0^T \left\langle E X^{x_0}(t), \Phi \, dB(t) \right\rangle_V + \frac{1}{4a} \left\langle Qz_2, z_2 \right\rangle_{L^2(D)}. \label{ito for z2}
\end{align}
\end{lemma}
\begin{proof}
Define the function $g : V \rightarrow \mathbb R$ by
\begin{equation}
g(x) = \left\langle Ex, x \right\rangle_V, \quad \forall x \in V.
\end{equation}

The It\^ o's formula (see e.g. \cite{dapratozabczyk}, Theorem 4.17.) is not applicable to the process $g(X^{x_0}(t))$ directly, because $(X^{x_0}(t), t \geq 0)$ 
is not a strong solution to the equation \eqref{linear equation with parameters}. We apply it to suitable finite--dimensional projections.

Let $\{ h_n, n \in \mathbb N \}$ be an orthonormal basis in $V$ consisting of elements from $\mathrm{Dom}(\mathcal A)$ and let $P_N$ be the operator of projection on the $\mathrm{span} \, \{ h_n, n = 1, \ldots N \}$, that is
$$
P_N x = \sum_{n=1}^N \left\langle x, h_n \right\rangle_V h_n, \quad \forall x \in V, \quad \forall N \in \mathbb N.
$$

Choose $N \in \mathbb N$ and set
$$
X^{x_0, N}(t) := P_N X^{x_0}(t), \quad \forall t \geq 0.
$$

The expansion for the $X^{x_0, N}(t)$ is finite, so $X_1^{x_0, N}(t) \in \text{Dom}(A)$, $X_2^{x_0, N}(t) \in \text{Dom}((-A)^{\frac{1}{2}})$ and consequently $X^{x_0, N}(t) \in \mathrm{Dom}(\mathcal A)$ for all $t \geq 0$. Now we may apply It\^ o's formula to the function $g(X^{x_0, N}(t))$, which yields
\begin{equation} \label{ito for z2-helpful}
dg(X^{x_0, N}(t)) = 2 \left\langle EX^{x_0, N}(t), dX^{x_0, N}(t) \right\rangle_V + \frac{1}{2} \, \mathrm{Tr} \left( 2E \Phi \Phi^* \right) \, dt,
\end{equation}
where $\mathrm{Tr} \, (\cdot)$ denotes the trace of the (nuclear) operator.

First, we simplify the second term by the following calculation
\begin{align}
\mathrm{Tr} \, (E \Phi \Phi^*) &= \mathrm{Tr} \, \left( \left( \begin{array}{cc} E_1 & E_2 \\ E_3 & E_4 \end{array} \right) \left( \begin{array}{cc} 0 & 0 \\ 0 & Q \end{array} \right) \right) = \mathrm{Tr} \, \left( \begin{array}{cc} 0 & E_2 Q \\ 0 & E_4 Q \end{array} \right) = \mathrm{Tr} \, (E_4 Q) \notag \\
&= \mathrm{Tr} \, \left( \sum_{k=1}^{\infty} z_{2,k}^2 E_{k,4} Q + \sum_{k=1}^{\infty} \sum_{l=k+1}^{\infty} \frac{z_{2,k} z_{2,l}}{D_{k,l}} E_{k,l,4} Q \right) \notag \\
&= \sum_{k=1}^{\infty} z_{2,k}^2 \mathrm{Tr} \, (E_{k,4} Q) + \sum_{k=1}^{\infty} \sum_{l=k+1}^{\infty} \frac{z_{2,k} z_{2,l}}{D_{k,l}} \mathrm{Tr} \, (E_{k,l,4} Q). \notag
\end{align}

Now we compute the partial traces $\mathrm{Tr} \, (E_{k,4} Q)$ and $\mathrm{Tr} \, (E_{k,l,4} Q)$. According to the definition of the trace
\begin{align}
\mathrm{Tr} \, (E_{k,4} Q) &= \sum_{j=1}^{\infty} \left\langle E_{k,4} Q e_j, e_j \right\rangle_{L^2(D)} = \sum_{j=1}^{\infty} \lambda_j \left\langle E_{k,4} e_j, e_j \right\rangle_{L^2(D)} \notag \\
&= \sum_{j=1}^{\infty} \lambda_j \left\langle e_j, e_k \right\rangle_{L^2(D)}^2 = \sum_{j=1}^{\infty} \lambda_j \delta_{j,k} = \lambda_k, \notag
\end{align}
where $\delta$ stands for the Kronecker's delta. Similarly, we have
\begin{align}
\mathrm{Tr} \, (E_{k,l,4} Q) &= \sum_{j=1}^{\infty} \left\langle E_{k,l,4} Q e_j, e_j \right\rangle_{L^2(D)} \notag \\
&= \sum_{j=1}^{\infty} \lambda_j \left( 8a^2 (\alpha_k + \alpha_l) \left\langle e_j, e_k \right\rangle_{L^2(D)} \left\langle e_l, e_j \right\rangle_{L^2(D)} \right. \notag \\
&\phantom{=} \left. + 8a^2 (\alpha_k + \alpha_l) \left\langle e_j, e_l \right\rangle_{L^2(D)} \left\langle e_k, e_j \right\rangle_{L^2(D)} \right) \notag \\
&= \sum_{j=1}^{\infty} 16a^2 \lambda_j (\alpha_k + \alpha_l) \delta_{j,k} \delta_{j,l} = 0, \notag
\end{align}
since $k \neq l$ in \eqref{definition of E}. Therefore
\begin{equation}
\mathrm{Tr} \, (E \Phi \Phi^*) = \sum_{k=1}^{\infty} \lambda_k z_{2,k}^2 = \left\langle Q z_2, z_2 \right\rangle_{L^2(D)}.
\end{equation}

Using this formulae and Lemma \ref{properties of E}, the expression \eqref{ito for z2-helpful} implies
\begin{align}
dg(X^{x_0, N}(t)) &= 2 \left\langle E X^{x_0, N}(t), \mathcal A X^{x_0, N}(t) \right\rangle_V \, dt + 2 \left\langle E X^{x_0, N}(t), \Phi \, dB(t) \right\rangle_V \notag \\
&\phantom{=}+ \left\langle Q z_2, z_2 \right\rangle_{L^2(D)} \, dt \notag \\
&= -4a \left\langle X_2^{x_0, N}(t), z_2 \right\rangle_{L^2(D)}^2 \, dt + 2 \left\langle E X^{x_0, N}(t), \Phi \, dB(t) \right\rangle_V \notag \\
&\phantom{=} + \left\langle Q z_2, z_2 \right\rangle_{L^2(D)} \, dt. \notag
\end{align}

Integrating previous formula over the interval $(0,T)$, we arrive at
\begin{align}
&\frac{1}{T} \int_0^T \left\langle X_2^{x_0, N}(t), z_2 \right\rangle_{L^2(D)}^2 \, dt = \notag \\
&= - \frac{1}{4aT} \left( \left\langle E X^{x_0, N}(T), X^{x_0, N}(T) \right\rangle_V - \left\langle E x_0^N, x_0^N \right\rangle_V \right) \notag \\
&\phantom{=}+ \frac{1}{2aT} \int_0^T \left\langle E X^{x_0, N}(t), \Phi \, dB(t) \right\rangle_V + \frac{1}{4a} \left\langle Qz_2, z_2 \right\rangle_{L^2(D)}. \label{ito for z2N}
\end{align}

Since
\begin{equation}
\left| \left\langle X_2^{x_0, N}(t), z_2 \right\rangle_{L^2(D)} \right| \leq \| X_2^{x_0}(t) \|_{L^2(D)} \, \| z_2 \|_{L^2(D)}, \quad \forall t \geq 0, \quad \forall N \in \mathbb N,
\end{equation}
we may use the random variable $\| X_2^{x_0}(t) \|_{L^2(D)}^2$ as an integrable majorant for the integral on the left--hand side. Also,
\begin{equation}
\int_0^T \left\langle E X^{x_0, N}(t), \Phi \, dB(t) \right\rangle_V \rightarrow \int_0^T \left\langle E X^{x_0}(t), \Phi \, dB(t) \right\rangle_V, \quad N \rightarrow \infty \, \, \mathrm{in} \, \, L^2(\Omega),
\end{equation}
because, for some positive constant $C >0$, we have that
\begin{equation}
\mathbb E \left| \int_0^T \left\langle E \left( X^{x_0, N}(t) - X^{x_0}(t) \right), \Phi \, dB(t) \right\rangle_V \right|^2 \leq C \int_0^T \mathbb E \left\| X^{x_0, N}(t) - X^{x_0}(t) \right\|_V^2 \, dt,
\end{equation}
which tends to $0$ as $N \rightarrow \infty$, since
\begin{equation}
X^{x_0, N}(t) \rightarrow X^{x_0}(t), \quad \forall t \geq 0, \quad N \rightarrow \infty \, \, \mathrm{in} \, \, L^2(\Omega; V).
\end{equation}

Hence we obtain \eqref{ito for z2} by passing $N$ to infinity in \eqref{ito for z2N}.
\end{proof}

We will also need the following Lemma for convergence of some cross terms to zero.

\begin{lemma} \label{cross terms go to zero}
1) Let $z_1 \in \mathrm{Dom}((-A)^{\frac{1}{2}})$ and $z_2 \in L^2(D)$ be arbitrary. Then
\begin{equation}
\lim_{T \rightarrow \infty} \frac{1}{T} \int_0^T \left\langle X_1^{x_0}(t), z_1 \right\rangle_{\mathrm{Dom}((-A)^{\frac{1}{2}})} \left\langle X_2^{x_0}(t), z_2 \right\rangle_{L^2(D)} \, dt = 0, \quad \mathbb P-a.s.
\end{equation}

2) Let $f_k, f_l \in \mathrm{Dom}((-A)^{\frac{1}{2}})$, $k \neq l$, be arbitrary. Then
\begin{equation}
\lim_{T \rightarrow \infty} \frac{1}{T} \int_0^T \left\langle X_1^{x_0}(t), f_k \right\rangle_{\mathrm{Dom}((-A)^{\frac{1}{2}})} \left\langle X_1^{x_0}(t), f_l \right\rangle_{\mathrm{Dom}((-A)^{\frac{1}{2}})} \, dt = 0, \quad \mathbb P-a.s.
\end{equation}

3) Let $e_k, e_l \in L^2(D)$, $k \neq l$, be arbitrary. Then
\begin{equation}
\lim_{T \rightarrow \infty} \frac{1}{T} \int_0^T \left\langle X_2^{x_0}(t), e_k \right\rangle_{L^2(D)} \left\langle X_2^{x_0}(t), e_l \right\rangle_{L^2(D)} \, dt = 0, \quad \mathbb P-a.s.
\end{equation}
\end{lemma}
\begin{proof}
1) Let $z_1 \in \mathrm{Dom}((-A)^{\frac{1}{2}})$ and $z_2 \in L^2(D)$. Using Theorem \ref{ergodic theorem} with a~fun\-cti\-onal $\varrho : V \rightarrow \mathbb R$, $\varrho (y) = \left\langle y, \left( \begin{array}{c} z_1 \\ 0 \end{array} \right) \right\rangle_V \left\langle y, \left( \begin{array}{c} 0 \\ z_2 \end{array} \right) \right\rangle_V$, $y \in V$ and $m = 1$, we get
\begin{align}
&\lim_{T \rightarrow \infty} \frac{1}{T} \int_0^T \left\langle X_1^{x_0}(t), z_1 \right\rangle_{\mathrm{Dom}((-A)^{\frac{1}{2}})} \left\langle X_2^{x_0}(t), z_2 \right\rangle_{L^2(D)} \, dt = \notag \\
&= \lim_{T \rightarrow \infty} \frac{1}{T} \int_0^T \left\langle X^{x_0}(t), \left( \begin{array}{c} z_1 \\ 0 \end{array} \right) \right\rangle_V \left\langle X^{x_0}(t), \left( \begin{array}{c} 0 \\ z_2 \end{array} \right) \right\rangle_V \, dt \notag \\
&= \int_V \left\langle y, \left( \begin{array}{c} z_1 \\ 0 \end{array} \right) \right\rangle_V \left\langle y, \left( \begin{array}{c} 0 \\ z_2 \end{array} \right) \right\rangle_V \, \mu_{\infty}^{(a,b)} (dy) \notag \\
&= \left\langle Q_{\infty}^{(a,b)} \left( \begin{array}{c} z_1 \\ 0 \end{array} \right), \left( \begin{array}{c} 0 \\ z_2 \end{array} \right) \right\rangle_V \notag \\
&= \left\langle \left( \begin{array}{cc} \frac{1}{4ab} Q & 0 \\ 0 & \frac{1}{4a} Q \end{array} \right) \left( \begin{array}{c} z_1 \\ 0 \end{array} \right), \left( \begin{array}{c} 0 \\ z_2 \end{array} \right) \right\rangle_V = 0, \quad \mathbb P-a.s. \notag
\end{align}

2) Let $f_k, f_l \in \mathrm{Dom}((-A)^{\frac{1}{2}})$ with $k \neq l$. Using Theorem \ref{ergodic theorem} with a~fun\-cti\-onal $\varrho : V \rightarrow \mathbb R$, $\varrho (y) = \left\langle y, \left( \begin{array}{c} f_k \\ 0 \end{array} \right) \right\rangle_V \left\langle y, \left( \begin{array}{c} f_l \\ 0 \end{array} \right) \right\rangle_V$, $y \in V$ and $m = 1$, we get
\begin{align}
&\lim_{T \rightarrow \infty} \frac{1}{T} \int_0^T \left\langle X_1^{x_0}(t), f_k \right\rangle_{\mathrm{Dom}((-A)^{\frac{1}{2}})} \left\langle X_1^{x_0}(t), f_l \right\rangle_{\mathrm{Dom}((-A)^{\frac{1}{2}})} \, dt = \notag \\
&= \lim_{T \rightarrow \infty} \frac{1}{T} \int_0^T \left\langle X^{x_0}(t), \left( \begin{array}{c} f_k \\ 0 \end{array} \right) \right\rangle_V \left\langle X^{x_0}(t), \left( \begin{array}{c} f_l \\ 0 \end{array} \right) \right\rangle_V \, dt \notag \\
&= \left\langle \left( \begin{array}{cc} \frac{1}{4ab} Q & 0 \\ 0 & \frac{1}{4a} Q \end{array} \right) \left( \begin{array}{c} f_k \\ 0 \end{array} \right), \left( \begin{array}{c} f_l \\ 0 \end{array} \right) \right\rangle_V \notag \\
&= \frac{1}{4ab} \left\langle Q f_k, f_l \right\rangle_{\mathrm{Dom}((-A)^{\frac{1}{2}})} \notag \\
&= \frac{\lambda_k}{4ab} \delta_{k,l} = 0, \quad \mathbb P-a.s. \notag
\end{align}

3) Let $e_k, e_l \in L^2(D)$ with $k \neq l$. The proof is analogous to the previous one with functional $\varrho : V \rightarrow \mathbb R$, $\varrho (y) = \left\langle y, \left( \begin{array}{c} 0 \\ e_k \end{array} \right) \right\rangle_V \left\langle y, \left( \begin{array}{c} 0 \\ e_l \end{array} \right) \right\rangle_V$, $y \in V$.
\end{proof}

We will also need the following Lemma from \cite{janak}.

\begin{lemma} \label{convergence to zero}
Let $(X^{x_0}(t), t \geq 0)$ be a solution to the linear equation \eqref{linear equation with parameters} and $R \in \mathcal L(V)$. Then
$$
\frac{1}{\sqrt{t}} \left\langle RX^{x_0}(t), X^{x_0}(t) \right\rangle_V \rightarrow 0
$$
in $L^1(\Omega)$ as $t \rightarrow \infty$.
\end{lemma}
\begin{proof}
See \cite{janak}.
\end{proof}

Asymptotic normality of the estimator $\bar{a}_{T, z_2}$ is formulated in the following Theorem.

\begin{theorem} \label{asymptotic normality of abar}
Let $0 \neq z_2 \in L^2(D)$ be arbitrary. The estimator $\bar{a}_{T, z_2}$ is asymp\-to\-ti\-cally normal, i.e.,
\begin{align}
&\mathrm{Law} \left( \sqrt{T} \left( \bar{a}_{T, z_2} - a \right) \right) \stackrel{w^*}{\longrightarrow} \notag \\
&\phantom{=}N \left( 0, \frac{8a^3}{\left\langle Q z_2, z_2 \right\rangle_{L^2(D)}^2} \sum_{k=1}^{\infty} \sum_{n=1}^{\infty} \frac{\lambda_k \lambda_n (\alpha_k + \alpha_n) z_{2,k}^2 z_{2,n}^2}{b(\alpha_k - \alpha_n)^2 + 8a^2 (\alpha_k + \alpha_n)} \right), \quad T \rightarrow \infty. \notag
\end{align}
\end{theorem}
\begin{proof}
Set $0 \neq z_2 \in L^2(D)$. If we use formula \eqref{abar} for the estimator $\bar{a}_{T, z_2}$ and Lemma \ref{ito for z2-lemma}, we obtain
\begin{align}
& \quad \quad \sqrt{T} \left( \bar{a}_{T, z_2} - a \right) = \label{asymptotics for abarz2} \\
&= \frac{\sqrt{T}}{\frac{4}{T} \int_0^T \left\langle X_2^{x_0}(t), z_2 \right\rangle_{L^2(D)}^2 \, dt} \left( \left\langle Q z_2, z_2 \right\rangle_{L^2(D)} - 4a \frac{1}{T} \int_0^T \left\langle X_2^{x_0}(t), z_2 \right\rangle_{L^2(D)}^2 \, dt \right) \notag \\
&= \frac{\sqrt{T}}{\frac{4}{T} \int_0^T \left\langle X_2^{x_0}(t), z_2 \right\rangle_{L^2(D)}^2 \, dt} \left( \frac{1}{T} \left( \left\langle E X^{x_0}(T), X^{x_0}(T) \right\rangle_V - \left\langle E x_0, x_0 \right\rangle_V \right) \right. \notag \\
&\phantom{=} \left. - \frac{2}{T} \int_0^T \left\langle E X^{x_0}(t), \Phi \, dB(t) \right\rangle_V \right) \notag \\
&= \frac{1}{\frac{4}{T} \int_0^T \left\langle X_2^{x_0}(t), z_2 \right\rangle_{L^2(D)}^2 \, dt} \frac{1}{\sqrt{T}} \left( \left\langle E X^{x_0}(T), X^{x_0}(T) \right\rangle_V - \left\langle E x_0, x_0 \right\rangle_V \right) \notag \\
&\phantom{=} - \frac{1}{\frac{2}{T} \int_0^T \left\langle X_2^{x_0}(t), z_2 \right\rangle_{L^2(D)}^2 \, dt} \frac{1}{\sqrt{T}} \int_0^T \left\langle E X^{x_0}(t), \Phi \, dB(t) \right\rangle_V. \notag
\end{align}

The first term on the right--hand side converges to zero in probability as $T \rightarrow \infty$, since
$$
\lim_{T \rightarrow \infty} \frac{1}{T} \int_0^T \left\langle X_2^{x_0}(t), z_2 \right\rangle_{L^2(D)}^2 \, dt = \frac{1}{4a} \left\langle Q z_2, z_2 \right\rangle_{L^2(D)}, \quad \mathbb P-a.s.
$$
by Theorem \ref{abarbbar - general case} and
$$
\lim_{T \rightarrow \infty} \frac{1}{\sqrt{T}} \left( \left\langle E X^{x_0}(T), X^{x_0}(T) \right\rangle_V - \left\langle E x_0, x_0 \right\rangle_V \right) = 0, \quad \mathrm{in} \, \, L^1(\Omega)
$$
by Lemma \ref{convergence to zero}. Define
\begin{align}
w(T) &= \frac{1}{\sqrt{T}} \int_0^T \left\langle E X^{x_0}(t), \Phi \, dB(t) \right\rangle_V \notag \\
&= \frac{1}{\sqrt{T}} \int_0^T \sum_{n=1}^{\infty} \sqrt{\lambda_n} \left\langle E X^{x_0}(t), \left( \begin{array}{c} 0 \\ e_n \end{array} \right) \right\rangle_V \, d\beta_n (t), \notag
\end{align}
where we have used the representation of $V$--valued Brownian motion $B(t)$. For any $n \in \mathbb N$ $\beta_n(t) = \left\langle B(t), e_n \right\rangle_V$ are mutually independent scalar Brownian motions (see \cite{dapratozabczyk}).

First, let us express scalar product in the above series
\begin{equation}
\left\langle E X^{x_0}(t), \left( \begin{array}{c} 0 \\ e_n \end{array} \right) \right\rangle_V = \left\langle E_3 X_1^{x_0}(t) + E_4 X_2^{x_0}(t), e_n \right\rangle_{L^2(D)}.
\end{equation}

Next we have
\begin{align}
\left\langle E_3 X_1^{x_0}(t), e_n \right\rangle_{L^2(D)} &= \sum_{k=1}^{\infty} \sum_{l=k+1}^{\infty} \frac{z_{2,k} z_{2,l}}{D_{k,l}} \left\langle E_{k,l,3} X_1^{x_0}(t), e_n \right\rangle_{L^2(D)} \notag \\
&= \sum_{k=1}^{\infty} \sum_{l=k+1}^{\infty} \frac{z_{2,k} z_{2,l}}{D_{k,l}} \left( 4ab \alpha_l (\alpha_k - \alpha_l) \left\langle X_1^{x_0}(t), e_l \right\rangle_{L^2(D)} \delta_{k,n} \right. \notag \\
&\phantom{=} \left. + 4ab \alpha_k (\alpha_l - \alpha_k) \left\langle X_1^{x_0}(t), e_k \right\rangle_{L^2(D)} \delta_{n,l} \right) \notag \\
&= \sum_{k=1}^{\infty} \sum_{l=k+1}^{\infty} \frac{z_{2,k} z_{2,l}}{D_{k,l}} 4ab \alpha_l (\alpha_k - \alpha_l) \left\langle X_1^{x_0}(t), e_l \right\rangle_{L^2(D)} \delta_{k,n} \notag \\
&\phantom{=}+ \sum_{k=1}^{\infty} \sum_{l=k+1}^{\infty} \frac{z_{2,k} z_{2,l}}{D_{k,l}} 4ab \alpha_k (\alpha_l - \alpha_k) \left\langle X_1^{x_0}(t), e_k \right\rangle_{L^2(D)} \delta_{n,l} \notag \\
&= (I) + (II), \notag
\end{align}
where we have used the fact that $E_{k,3} = 0$. Furthermore, compute
\begin{align}
&\left\langle E_4 X_2^{x_0}(t), e_n \right\rangle_{L^2(D)} = \notag \\
&=\left\langle \sum_{k=1}^{\infty} z_{2,k}^2 E_{k,4} X_2^{x_0}(t) + \sum_{k=1}^{\infty} \sum_{l=k+1}^{\infty} \frac{z_{2,k} z_{2,l}}{D_{k,l}} E_{k,l,4} X_2^{x_0}(t), e_n \right\rangle_{L^2(D)} \notag \\
&= \sum_{k=1}^{\infty} z_{2,k}^2 \left\langle X_2^{x_0}(t), e_k \right\rangle_{L^2(D)} \delta_{k,n} \notag \\
&\phantom{=}+ \sum_{k=1}^{\infty} \sum_{l=k+1}^{\infty} \frac{z_{2,k} z_{2,l}}{D_{k,l}} 8 a^2 (\alpha_k + \alpha_l) \left\langle X_2^{x_0}(t), e_k \right\rangle_{L^2(D)} \delta_{n,l} \notag \\
&\phantom{=} + \sum_{k=1}^{\infty} \sum_{l=k+1}^{\infty} \frac{z_{2,k} z_{2,l}}{D_{k,l}} 8 a^2 (\alpha_k + \alpha_l) \left\langle X_2^{x_0}(t), e_l \right\rangle_{L^2(D)} \delta_{k,n} \notag \\
&= (III) + (IV) + (V). \notag
\end{align}

By the central limit theorem for martingales (see e.g. \cite{kutoyants}, Proposition 1.22.), $\mathrm{Law} \left( w(T) \right)$ converges weakly to a Gaussian distribution with a zero mean and variance given by the $\mathbb P-a.s.$ limit
\begin{equation}
\lim_{T \rightarrow \infty} \frac{1}{T} \int_0^T \sum_{n=1}^{\infty} \lambda_n \left( (I) + \ldots + (V) \right)^2 \, dt.
\end{equation}

Therefore all the limits in the rest of the proof are considered in the $\mathbb P-a.s.$ sense.

The limits of the cross terms are zero by Lemma \ref{cross terms go to zero}. For example
\begin{align}
&\lim_{T \rightarrow \infty} \frac{1}{T} \int_0^T \sum_{n=1}^{\infty} \lambda_n (I)(II) \, dt = \notag \\
&= \lim_{T \rightarrow \infty} \frac{1}{T} \int_0^T \sum_{n=1}^{\infty} \lambda_n \left( \sum_{l=n+1}^{\infty} \frac{z_{2,n} z_{2,l}}{D_{n,l}} 4ab \alpha_l (\alpha_n - \alpha_l) \left\langle X_1^{x_0}(t), e_l \right\rangle_{L^2(D)} \right) \times \notag \\
&\phantom{=} \left( \sum_{k=1}^{n-1} \frac{z_{2,k} z_{2,n}}{D_{k,n}} 4ab \alpha_k (\alpha_n - \alpha_k) \left\langle X_1^{x_0}(t), e_k \right\rangle_{L^2(D)} \right) \, dt \notag \\
&= \sum_{n=1}^{\infty} \sum_{l=n+1}^{\infty} \sum_{k=1}^{n-1} \lambda_n \frac{z_{2,n}^2 z_{2,l} z_{2,k}}{D_{n,l} D_{k,n}} 16 a^2 b^2 \alpha_l \alpha_k (\alpha_n - \alpha_l) (\alpha_n - \alpha_k) \times \notag \\
&\phantom{=} \left( \lim_{T \rightarrow \infty} \frac{1}{T} \int_0^T \left\langle X_1^{x_0}(t), e_l \right\rangle_{L^2(D)} \left\langle X_1^{x_0}(t), e_k \right\rangle_{L^2(D)} \, dt \right) = 0, \notag
\end{align}
because $k \neq l$, $\left\langle X_1^{x_0}(t), e_k \right\rangle_{L^2(D)} = \frac{1}{\sqrt{\alpha_k}} \left\langle X_1^{x_0}(t), f_k \right\rangle_{\mathrm{Dom}((-A)^{\frac{1}{2}})}$ (and similarly with the index $l$) and we may use Lemma \ref{cross terms go to zero}, 2).

Also
\begin{align}
&\lim_{T \rightarrow \infty} \frac{1}{T} \int_0^T \sum_{n=1}^{\infty} \lambda_n (I)(III) \, dt = \notag \\
&= \lim_{T \rightarrow \infty} \frac{1}{T} \int_0^T \sum_{n=1}^{\infty} \lambda_n \left( \sum_{l=n+1}^{\infty} \frac{z_{2,n} z_{2,l}}{D_{n,l}} 4ab \alpha_l (\alpha_n - \alpha_l) \left\langle X_1^{x_0}(t), e_l \right\rangle_{L^2(D)} \right) \times \notag \\
&\phantom{=} \left( z_{2,n}^2 \left\langle X_2^{x_0}(t), e_n \right\rangle_{L^2(D)} \right) \, dt \notag \\
&= \sum_{n=1}^{\infty} \sum_{l=n+1}^{\infty} \lambda_n \frac{z_{2,n}^3 z_{2,l}}{D_{n,l}} 4ab \alpha_l (\alpha_n - \alpha_l) \times \notag \\
&\phantom{=} \left( \lim_{T \rightarrow \infty} \frac{1}{T} \int_0^T \left\langle X_1^{x_0}(t), e_l \right\rangle_{L^2(D)} \left\langle X_2^{x_0}(t), e_n \right\rangle_{L^2(D)} \, dt \right) = 0, \notag
\end{align}
since $n \neq l$ and the limit follows from Lemma \ref{cross terms go to zero}, 1). The remaining limits of the cross terms are handled similarly.

Now we compute the limits of the "diagonal" terms.
\begin{align}
(A) &= \lim_{T \rightarrow \infty} \frac{1}{T} \int_0^T \sum_{n=1}^{\infty} \lambda_n (I)^2 \, dt \notag \\
&= \lim_{T \rightarrow \infty} \frac{1}{T} \int_0^T \sum_{n=1}^{\infty} \lambda_n \left( \sum_{l=n+1}^{\infty} \frac{z_{2,n} z_{2,l}}{D_{n,l}} 4ab \alpha_l (\alpha_n - \alpha_l) \left\langle X_1^{x_0}(t), e_l \right\rangle_{L^2(D)} \right)^2 \, dt \notag \\
&\stackrel{(*)}{=} \lim_{T \rightarrow \infty} \frac{1}{T} \int_0^T \sum_{n=1}^{\infty} \sum_{l=n+1}^{\infty} \lambda_n \frac{z_{2,n}^2 z_{2,l}^2}{D_{n,l}^2} 16 a^2 b^2 \alpha_l^2 (\alpha_n - \alpha_l)^2 \left\langle X_1^{x_0}(t), e_l \right\rangle_{L^2(D)}^2 \, dt \notag \\
&= \sum_{n=1}^{\infty} \sum_{l=n+1}^{\infty} \frac{1}{D_{n,l}^2} 4ab \lambda_n \lambda_l \alpha_l (\alpha_n - \alpha_l)^2 z_{2,n}^2 z_{2,l}^2, \notag
\end{align}
since $\lim_{T \rightarrow \infty} \frac{1}{T} \int_0^T \left\langle X_1^{x_0}(t), e_l \right\rangle_{L^2(D)}^2 \, dt = \frac{\lambda_l}{4ab \alpha_l}$ and in the equality $(*)$ we have also used Lemma \ref{cross terms go to zero} for the cross summands.

In a similar manner, we have
\begin{align}
(B) &= \lim_{T \rightarrow \infty} \frac{1}{T} \int_0^T \sum_{n=1}^{\infty} \lambda_n (II)^2 \, dt \notag \\
&= \lim_{T \rightarrow \infty} \frac{1}{T} \int_0^T \sum_{n=1}^{\infty} \lambda_n \left( \sum_{k=1}^{n-1} \frac{z_{2,k} z_{2,n}}{D_{k,n}} 4ab \alpha_k (\alpha_n - \alpha_k) \left\langle X_1^{x_0}(t), e_k \right\rangle_{L^2(D)} \right)^2 \, dt \notag \\
&\stackrel{(*)}{=} \lim_{T \rightarrow \infty} \frac{1}{T} \int_0^T \sum_{n=1}^{\infty} \sum_{k=1}^{n-1} \lambda_n \frac{z_{2,k}^2 z_{2,n}^2}{D_{k,n}^2} 16 a^2 b^2 \alpha_k^2 (\alpha_n - \alpha_k)^2 \left\langle X_1^{x_0}(t), e_k \right\rangle_{L^2(D)}^2 \, dt \notag \\
&= \sum_{n=1}^{\infty} \sum_{k=1}^{n-1} \frac{1}{D_{k,n}^2} 4ab \lambda_n \lambda_k \alpha_k (\alpha_n - \alpha_k)^2 z_{2,k}^2 z_{2,n}^2, \notag
\end{align}
\begin{align}
(C) &= \lim_{T \rightarrow \infty} \frac{1}{T} \int_0^T \sum_{n=1}^{\infty} \lambda_n (III)^2 \, dt \notag \\
&= \lim_{T \rightarrow \infty} \frac{1}{T} \int_0^T \sum_{n=1}^{\infty} \lambda_n z_{2,n}^4 \left\langle X_2^{x_0}(t), e_n \right\rangle_{L^2(D)}^2 \, dt \notag \\
&= \frac{1}{4a} \sum_{n=1}^{\infty} \lambda_n^2 z_{2,n}^4, \notag
\end{align}
since $\lim_{T \rightarrow \infty} \frac{1}{T} \int_0^T \left\langle X_2^{x_0}(t), e_n \right\rangle_{L^2(D)}^2 \, dt = \frac{\lambda_n}{4a}$.

Next, we have
\begin{align}
(D) &= \lim_{T \rightarrow \infty} \frac{1}{T} \int_0^T \sum_{n=1}^{\infty} \lambda_n (IV)^2 \, dt \notag \\
&= \lim_{T \rightarrow \infty} \frac{1}{T} \int_0^T \sum_{n=1}^{\infty} \lambda_n \left( \sum_{k=1}^{n-1} \frac{z_{2,k} z_{2,n}}{D_{k,n}} 8a^2 (\alpha_k + \alpha_n) \left\langle X_2^{x_0}(t), e_k \right\rangle_{L^2(D)} \right)^2 \, dt \notag \\
&\stackrel{(*)}{=} \lim_{T \rightarrow \infty} \frac{1}{T} \int_0^T \sum_{n=1}^{\infty} \sum_{k=1}^{n-1} \lambda_n \frac{z_{2,k}^2 z_{2,n}^2}{D_{k,n}^2} 64 a^4 (\alpha_k + \alpha_n)^2 \left\langle X_2^{x_0}(t), e_k \right\rangle_{L^2(D)}^2 \, dt \notag \\
&= \sum_{n=1}^{\infty} \sum_{k=1}^{n-1} \frac{1}{D_{k,n}^2} 16a^3 \lambda_n \lambda_k (\alpha_k + \alpha_n)^2 z_{2,k}^2 z_{2,n}^2, \notag
\end{align}
\begin{align}
(E) &= \lim_{T \rightarrow \infty} \frac{1}{T} \int_0^T \sum_{n=1}^{\infty} \lambda_n (V)^2 \, dt \notag \\
&= \lim_{T \rightarrow \infty} \frac{1}{T} \int_0^T \sum_{n=1}^{\infty} \lambda_n \left( \sum_{l=n+1}^{\infty} \frac{z_{2,n} z_{2,l}}{D_{n,l}} 8a^2 (\alpha_n + \alpha_l) \left\langle X_2^{x_0}(t), e_l \right\rangle_{L^2(D)} \right)^2 \, dt \notag \\
&\stackrel{(*)}{=} \lim_{T \rightarrow \infty} \frac{1}{T} \int_0^T \sum_{n=1}^{\infty} \sum_{l=n+1}^{\infty} \lambda_n \frac{z_{2,n}^2 z_{2,l}^2}{D_{n,l}^2} 64 a^4 (\alpha_n + \alpha_l)^2 \left\langle X_2^{x_0}(t), e_l \right\rangle_{L^2(D)}^2 \, dt \notag \\
&= \sum_{n=1}^{\infty} \sum_{l=n+1}^{\infty} \frac{1}{D_{n,l}^2} 16a^3 \lambda_n \lambda_l (\alpha_n + \alpha_l)^2 z_{2,n}^2 z_{2,l}^2. \notag
\end{align}

The resulting formula for the limiting variance of $w(T)$ is the sum of the five above terms, however it may be further simplified. Since
\begin{equation}
\sum_{n=1}^{\infty} \sum_{k=1}^{n-1} a_{n,k} = \sum_{k=1}^{\infty} \sum_{n=k+1}^{\infty} a_{n,k}, \quad \forall a_{n,k} \in \mathbb R,
\end{equation}
we may switch the sums in the term $(B)$ and by changing indices $n \mapsto k$, $l \mapsto n$ in the term $(A)$, we arrive at
\begin{equation}
(A) + (B) = \sum_{k=1}^{\infty} \sum_{n=k+1}^{\infty} \frac{1}{D_{k,n}^2} 4ab \lambda_k \lambda_n (\alpha_k + \alpha_n) (\alpha_n - \alpha_k)^2 z_{2,k}^2 z_{2,n}^2.
\end{equation}

Similary, if we switch the sums in the term $(D)$ and change indices $n \mapsto k$, $l \mapsto n$ in the term $(E)$, we find out that $(D) = (E)$, so
\begin{equation}
(D) + (E) = \sum_{k=1}^{\infty} \sum_{n=k+1}^{\infty} \frac{1}{D_{k,n}^2} 32a^3 \lambda_k \lambda_n (\alpha_k + \alpha_n)^2 z_{2,k}^2 z_{2,n}^2
\end{equation}
and consequently
\begin{align}
& \quad (A) + (B) + (D) + (E) = \label{ABDE} \\
&= \sum_{k=1}^{\infty} \sum_{n=k+1}^{\infty} \frac{1}{D_{k,n}^2} 4a \lambda_k \lambda_n (\alpha_k + \alpha_n) z_{2,k}^2 z_{2,n}^2 \left( b(\alpha_n - \alpha_k)^2 + 8a^2 (\alpha_k + \alpha_n) \right) \notag \\
&= \sum_{k=1}^{\infty} \sum_{n=k+1}^{\infty} \frac{1}{D_{k,n}} 4a \lambda_k \lambda_n (\alpha_k + \alpha_n) z_{2,k}^2 z_{2,n}^2. \notag
\end{align}

Since
\begin{equation}
\sum_{k=1}^{\infty} \sum_{n=k+1}^{\infty} a_{n,k} = \frac{1}{2} \sum_{k=1}^{\infty} \sum_{n=1, n \neq k}^{\infty} a_{n,k}, \quad \mathrm{if} \, \, \forall k \in \mathbb N \, \forall n \in \mathbb N \quad a_{k,n} = a_{n,k},
\end{equation}
the sum on the right--hand side of \eqref{ABDE} equals to
\begin{equation} \label{ABDE simplified}
2a \sum_{k=1}^{\infty} \sum_{n=1, n \neq k}^{\infty} \frac{1}{D_{k,n}} \lambda_k \lambda_n (\alpha_k + \alpha_n) z_{2,k}^2 z_{2,n}^2.
\end{equation}

The summand $(C)$ is the corresponding sum to \eqref{ABDE simplified}, where $n = k$, so we may add it and we end up with the formula for the limiting variance of $w(T)$, that is
\begin{equation}
\mathrm{Var} \left( w(T) \right) = 2a \sum_{k=1}^{\infty} \sum_{n=1}^{\infty} \frac{1}{D_{k,n}} \lambda_k \lambda_n (\alpha_k + \alpha_n) z_{2,k}^2 z_{2,n}^2, \quad T \rightarrow \infty.
\end{equation}

Since the multiplicative factor $- \frac{1}{\frac{2}{T} \int_0^T \left\langle X_2^{x_0}(t), z_2 \right\rangle_{L^2(D)}^2 \, dt}$ of $w(T)$ on the right--hand side of \eqref{asymptotics for abarz2} converges to $\frac{-2a}{\left\langle Q z_2, z_2 \right\rangle_{L^2(D)}}$ as $T \rightarrow \infty$, we arrive at
\begin{align}
&\mathrm{Law} \left( \sqrt{T} \left( \bar{a}_{T, z_2} - a \right) \right) \stackrel{w^*}{\longrightarrow} \notag \\
&\phantom{=}N \left( 0, \frac{8a^3}{\left\langle Q z_2, z_2 \right\rangle_{L^2(D)}^2} \sum_{k=1}^{\infty} \sum_{n=1}^{\infty} \frac{\lambda_k \lambda_n (\alpha_k + \alpha_n) z_{2,k}^2 z_{2,n}^2}{b(\alpha_k - \alpha_n)^2 + 8a^2 (\alpha_k + \alpha_n)} \right), \quad T \rightarrow \infty. \notag
\end{align}
\end{proof}

\subsection{Asymptotic normality of the estimator $\bar{b}_{T, z_1, z_2}$}
The estimator $\bar{b}_{T, z_1, z_2}$ (defined by \eqref{bbar}) is also asymptotically normal. The proof uses similiar technique as the proof of Theorem \ref{asymptotic normality of abar}, so the setup and auxiliary Lemmas will be analogous to those in previous subsection.

Let $k \in \mathbb N$ be arbitrary and define the operator $F_k : V \rightarrow V$ by
\begin{equation}
F_k x = F_k \left( \begin{array}{c} x_1 \\ x_2 \end{array} \right) = \left( \begin{array}{cc} F_{k,1} & F_{k,2} \\ F_{k,3} & F_{k,4} \end{array} \right) \left( \begin{array}{c} x_1 \\ x_2 \end{array} \right), \quad \forall x = \left( \begin{array}{c} x_1 \\ x_2 \end{array} \right) \in V,
\end{equation}
where
\begin{align}
F_{k,1} : x_1 &\longmapsto \left( b + \frac{4a^2}{\alpha_k} \right) \left\langle x_1, e_k \right\rangle_{L^2(D)} e_k, \notag \\
F_{k,2} : x_2 &\longmapsto \frac{2a}{\alpha_k} \left\langle x_2, e_k \right\rangle_{L^2(D)} e_k, \notag \\
F_{k,3} : x_1 &\longmapsto 2a \left\langle x_1, e_k \right\rangle_{L^2(D)} e_k, \notag \\
F_{k,4} : x_2 &\longmapsto \left\langle x_2, e_k \right\rangle_{L^2(D)} e_k, \notag
\end{align}
for any $x_1 \in \mathrm{Dom}((-A)^{\frac{1}{2}})$ and $x_2 \in L^2(D)$.

Let $k, l \in \mathbb N$ be arbitrary and define the operator $F_{k,l} : V \rightarrow V$ by
\begin{equation}
F_{k,l} x = F_k \left( \begin{array}{c} x_1 \\ x_2 \end{array} \right) = \frac{1}{D_{k,l}} \left( \begin{array}{cc} F_{k,l,1} & F_{k,l,2} \\ F_{k,l,3} & F_{k,l,4} \end{array} \right) \left( \begin{array}{c} x_1 \\ x_2 \end{array} \right), \quad \forall x = \left( \begin{array}{c} x_1 \\ x_2 \end{array} \right) \in V,
\end{equation}
where
\begin{align}
F_{k,l,1} : x_1 \longmapsto \, &8a^2 \sqrt{\frac{\alpha_l}{\alpha_k}} \left( 8a^2 + b (\alpha_k + \alpha_l) \right) \left\langle x_1, e_l \right\rangle_{L^2(D)} e_k \notag \\
&+ 8a^2 \sqrt{\frac{\alpha_k}{\alpha_l}} \left( 8a^2 + b (\alpha_k + \alpha_l) \right) \left\langle x_1, e_k \right\rangle_{L^2(D)} e_l, \notag \\
F_{k,l,2} : x_2 \longmapsto \, &4a \sqrt{\frac{\alpha_k}{\alpha_l}} \left( 8a^2 + b(\alpha_k - \alpha_l) \right) \left\langle x_2, e_k \right\rangle_{L^2(D)} e_l \notag \\
&+ 4a \sqrt{\frac{\alpha_l}{\alpha_k}} \left( 8a^2 + b(\alpha_l - \alpha_k) \right) \left\langle x_2, e_l \right\rangle_{L^2(D)} e_k, \notag \\
F_{k,l,3} : x_1 \longmapsto \, &4a \sqrt{\alpha_k \alpha_l} \left( 8a^2 + b(\alpha_k - \alpha_l) \right) \left\langle x_1, e_l \right\rangle_{L^2(D)} e_k \notag \\
&+ 4a \sqrt{\alpha_k \alpha_l} \left( 8a^2 + b(\alpha_l - \alpha_k) \right) \left\langle x_1, e_k \right\rangle_{L^2(D)} e_l, \notag \\
F_{k,l,4} : x_2 \longmapsto \, &16a^2 \sqrt{\alpha_k \alpha_l} \left\langle x_2, e_k \right\rangle_{L^2(D)} e_l + 16a^2 \sqrt{\alpha_k \alpha_l} \left\langle x_2, e_l \right\rangle_{L^2(D)} e_k, \notag
\end{align}
for any $x_1 \in \mathrm{Dom}((-A)^{\frac{1}{2}})$ and $x_2 \in L^2(D)$ with $D_{k,l}$ defined above.

The properties of the operators $F_k$ and $F_{k,l}$ needed in the sequel are summarized in the following Lemma.

\begin{lemma}
1) The operator $F_k \in \mathcal L(V)$ is self--adjoint for any given $k \in \mathbb N$. Moreover,
\begin{equation} \label{FkxAx}
\left\langle F_k x, \mathcal A x \right\rangle_V = -2ab \left\langle x_1, f_k \right\rangle_{\mathrm{Dom}((-A)^{\frac{1}{2}})}^2,
\end{equation}
for any $x = (x_1, x_2)^{\top} \in \mathrm{Dom}(\mathcal A)$.

2) The operator $F_{k,l} \in \mathcal L(V)$ is self--adjoint for any given $k, l \in \mathbb N$. Moreover,
\begin{equation} \label{FklxAx}
\left\langle F_{k,l} x, \mathcal A x \right\rangle_V = -4ab \left\langle x_1, f_k \right\rangle_{\mathrm{Dom}((-A)^{\frac{1}{2}})} \left\langle x_1, f_l \right\rangle_{\mathrm{Dom}((-A)^{\frac{1}{2}})},
\end{equation}
for any $x = (x_1, x_2)^{\top} \in \mathrm{Dom}(\mathcal A)$.
\end{lemma}
\begin{proof}
1) Let $k \in \mathbb N$ be arbitrary. Obviously $F_k \in \mathcal L (V)$ and for $x = (x_1, x_2)^{\top} \in V$ and $y = (y_1, y_2)^{\top} \in V$ we have
\begin{align}
\left\langle F_k x, y \right\rangle_V &= \left\langle \left( \begin{array}{c} F_{k,1} x_1 + F_{k,2} x_2 \\ F_{k,3} x_1 + F_{k,4} x_2 \end{array} \right), \left( \begin{array}{c} y_1 \\ y_2 \end{array} \right) \right\rangle_V \notag \\
&= \left( b + \frac{4a^2}{\alpha_k} \right) \left\langle x_1, e_k \right\rangle_{L^2(D)} \left\langle y_1, e_k \right\rangle_{\mathrm{Dom}((-A)^{\frac{1}{2}})} \notag \\
&\phantom{=}+\frac{2a}{\alpha_k} \left\langle x_2, e_k \right\rangle_{L^2(D)} \left\langle y_1, e_k \right\rangle_{\mathrm{Dom}((-A)^{\frac{1}{2}})} \notag \\
&\phantom{=}+ 2a \left\langle x_1, e_k \right\rangle_{L^2(D)} \left\langle y_2, e_k \right\rangle_{L^2(D)} + \left\langle x_2, e_k \right\rangle_{L^2(D)} \left\langle y_2, e_k \right\rangle_{L^2(D)} \notag \\
&= (b \alpha_k + 4a^2) \left\langle x_1, e_k \right\rangle_{L^2(D)} \left\langle y_1, e_k \right\rangle_{L^2(D)} + 2a \left\langle x_2, e_k \right\rangle_{L^2(D)} \left\langle y_1, e_k \right\rangle_{L^2(D)} \notag \\
&\phantom{=}+ 2a \left\langle x_1, e_k \right\rangle_{L^2(D)} \left\langle y_2, e_k \right\rangle_{L^2(D)} + \left\langle x_2, e_k \right\rangle_{L^2(D)} \left\langle y_2, e_k \right\rangle_{L^2(D)} \notag \\
&= \left\langle x, F_k y \right\rangle_V, \notag
\end{align}
hence $F_k = F_k^*$. Moreover, for every $x = (x_1, x_2)^{\top} \in \mathrm{Dom}(\mathcal A)$ we have
\begin{align}
\left\langle F_k x, \mathcal A x \right\rangle_V &= \left\langle \left( \begin{array}{c} F_{k,1} x_1 + F_{k,2} x_2 \\ F_{k,3} x_1 + F_{k,4} x_2 \end{array} \right), \left( \begin{array}{c} x_2 \\ bAx_1 - 2ax_2 \end{array} \right) \right\rangle_V \notag \\
&= \left( b + \frac{4a^2}{\alpha_k} \right) \left\langle x_1, e_k \right\rangle_{L^2(D)} \left\langle (-A)^{\frac{1}{2}} x_2, (-A)^{\frac{1}{2}} e_k \right\rangle_{L^2(D)} \notag \\
&\phantom{=}+ \frac{2a}{\alpha_k} \left\langle x_2, e_k \right\rangle_{L^2(D)} \left\langle (-A)^{\frac{1}{2}} x_2, (-A)^{\frac{1}{2}} e_k \right\rangle_{L^2(D)} \notag \\
&\phantom{=}+ 2ab \left\langle x_1, e_k \right\rangle_{L^2(D)} \left\langle A x_1, e_k \right\rangle_{L^2(D)} - 4a^2 \left\langle x_1, e_k \right\rangle_{L^2(D)} \left\langle x_2, e_k \right\rangle_{L^2(D)} \notag \\
&\phantom{=}+ b \left\langle x_2, e_k \right\rangle_{L^2(D)} \left\langle A x_1, e_k \right\rangle_{L^2(D)} -2a \left\langle x_2, e_k \right\rangle_{L^2(D)}^2 \notag \\
&= 2ab \left\langle x_1, e_k \right\rangle_{L^2(D)} \left\langle A x_1, e_k \right\rangle_{L^2(D)} \notag \\
&= -2ab \left\langle x_1, f_k \right\rangle_{\mathrm{Dom}((-A)^{\frac{1}{2}})}^2.
\end{align}

2) Let $k,l \in \mathbb N$ be arbitrary. It is clear that $F_{k,l} \in \mathcal L (V)$ and similarly as above it is possible to verify that $F_{k,l} = F_{k,l}^*$ and that \eqref{FklxAx} holds true for any $x \in \mathrm{Dom}(\mathcal A)$.
\end{proof}

Choose $0 \neq z_1 \in \mathrm{Dom}((-A)^{\frac{1}{2}})$ taking the form
\begin{equation}
z_1 = \sum_{k=1}^{\infty} \left\langle z_1, f_k \right\rangle_{\mathrm{Dom}((-A)^{\frac{1}{2}})} f_k = \sum_{k=1}^{\infty} z_{1,k} f_k,
\end{equation}
that is, $\{ z_{1,k}, k \in \mathbb N \}$ is the set of coordinates of the element $z_1$ with respect to the orthonormal basis in $\mathrm{Dom}((-A)^{\frac{1}{2}})$. Finally, define the operator $F : V \rightarrow V$ by
\begin{equation} \label{definition of F}
F = \sum_{k=1}^{\infty} z_{1,k}^2 F_k + \sum_{k=1}^{\infty} \sum_{l=k+1}^{\infty} z_{1,k} z_{1,l} F_{k,l}.
\end{equation}

The properties of the operator $F$ are summarized in the following Lemma.

\begin{lemma} \label{properties of F}
The operator $F \in \mathcal L(V)$. Moreover, it is self--adjoint and
\begin{equation} \label{FxAx}
\left\langle F x, \mathcal A x \right\rangle_V = -2ab \left\langle x_1, z_1 \right\rangle_{\mathrm{Dom}((-A)^{\frac{1}{2}})}^2, \quad \forall x = \left( \begin{array}{c} x_1 \\ x_2 \end{array} \right) \in \mathrm{Dom}(\mathcal A).
\end{equation}
\end{lemma}
\begin{proof}
Using inequalities in the proof of Lemma \ref{properties of E}, it is possible to verify that $F \in \mathcal L(V)$. Moreover, the linear combination of the self--adjoint operators is also the self--adjoint operator, hence $F = F^*$.

The property \eqref{FxAx} comes from \eqref{FkxAx} and \eqref{FklxAx}. The proof is analogous to the proof of Lemma \ref{properties of E}.
\end{proof}

We will also need an alternative representation for the process \\ $\frac{1}{T} \int_0^T \left\langle X_1^{x_0}(t), z_1 \right\rangle_{\mathrm{Dom}((-A)^{\frac{1}{2}})}^2 \, dt$.

\begin{lemma} \label{ito for z1-lemma}
The process $\frac{1}{T} \int_0^T \left\langle X_1^{x_0}(t), z_1 \right\rangle_{\mathrm{Dom}((-A)^{\frac{1}{2}})}^2 \, dt$ admits the following representation
\begin{align}
\frac{1}{T} \int_0^T \left\langle X_1^{x_0}(t), z_1 \right\rangle_{\mathrm{Dom}((-A)^{\frac{1}{2}})}^2 \, dt &= - \frac{1}{4abT} \left( \left\langle F X^{x_0}(T), X^{x_0}(T) \right\rangle_V - \left\langle F x_0, x_0 \right\rangle_V \right) \notag \\
&\phantom{=} + \frac{1}{2abT} \int_0^T \left\langle F X^{x_0}(t), \Phi \, dB(t) \right\rangle_V \notag \\
&\phantom{=} + \frac{1}{4ab} \left\langle Qz_1, z_1 \right\rangle_{\mathrm{Dom}((-A)^{\frac{1}{2}})}. \label{ito for z1}
\end{align}
\end{lemma}
\begin{proof}
Define the function $g_1 : V \rightarrow \mathbb R$ by
\begin{equation}
g_1(x) = \left\langle F x, x \right\rangle_V, \quad \forall x \in V.
\end{equation}

The application of It\^ o's formula to the function $g_1(X^{x_0, N}(t))$ (we also have to use suitable projections, see the proof of Lemma \ref{ito for z2-lemma}), yields
\begin{equation} \label{ito for z1-helpful}
dg_1(X^{x_0, N}(t)) = 2 \left\langle FX^{x_0, N}(t), dX^{x_0, N}(t) \right\rangle_V + \frac{1}{2} \, \mathrm{Tr} \left( 2F \Phi \Phi^* \right) \, dt.
\end{equation}

We also start with simplification of the second term
\begin{align}
\mathrm{Tr} \, (F \Phi \Phi^*) &= \mathrm{Tr} \, (F_4 Q) \notag \\
&= \sum_{k=1}^{\infty} z_{1,k}^2 \mathrm{Tr} \, (F_{k,4} Q) + \sum_{k=1}^{\infty} \sum_{l=k+1}^{\infty} \frac{z_{1,k} z_{1,l}}{D_{k,l}} \mathrm{Tr} \, (F_{k,l,4} Q) \notag \\
&= \left\langle Q z_1, z_1 \right\rangle_{\mathrm{Dom}((-A)^{\frac{1}{2}})},
\end{align}
since $\mathrm{Tr} \, (F_{k,4} Q) = \lambda_k$, $\mathrm{Tr} \, (F_{k,l,4} Q) = 0$, because $k \neq l$, and
\begin{equation}
\sum_{k=1}^{\infty} \lambda_k z_{1,k}^2 = \left\langle Q z_1, z_1 \right\rangle_{\mathrm{Dom}((-A)^{\frac{1}{2}})}.
\end{equation}

Lemma \ref{properties of F} and the formula \eqref{ito for z1-helpful} imply
\begin{align}
dg_1(X^{x_0, N}(t)) &= 2 \left\langle F X^{x_0, N}(t), \mathcal A X^{x_0, N}(t) \right\rangle_V \, dt + 2 \left\langle F X^{x_0, N}(t), \Phi \, dB(t) \right\rangle_V \notag \\
&\phantom{=}+ \left\langle Q z_1, z_1 \right\rangle_{\mathrm{Dom}((-A)^{\frac{1}{2}})} \, dt \notag \\
&= -4ab \left\langle X_1^{x_0, N}(t), z_1 \right\rangle_{\mathrm{Dom}((-A)^{\frac{1}{2}})}^2 \, dt + 2 \left\langle F X^{x_0, N}(t), \Phi \, dB(t) \right\rangle_V \notag \\
&\phantom{=} + \left\langle Q z_1, z_1 \right\rangle_{\mathrm{Dom}((-A)^{\frac{1}{2}})} \, dt. \notag
\end{align}

By integrating the above formula over the interval $(0,T)$ and passing $N$ to infinity, we arrive at \eqref{ito for z1}.
\end{proof}

Denote
\begin{equation}
Q_1 := \left\langle Q z_1, z_1 \right\rangle_{\mathrm{Dom}((-A)^{\frac{1}{2}})}, \quad Q_2 := \left\langle Q z_2, z_2 \right\rangle_{L^2(D)}.
\end{equation}
Asymptotic normality of the estimator $\bar{b}_{T, z_1, z_2}$ is formulated in the following Theorem.

\begin{theorem} \label{asymptotic normality of bbar}
Let $0 \neq z_1 \in \mathrm{Dom}((-A)^{\frac{1}{2}})$, $0 \neq z_2 \in L^2(D)$ be arbitrary. The estimator $\bar{b}_{T, z_1, z_2}$ is asymp\-to\-ti\-cally normal, i.e., $\mathrm{Law} \left( \sqrt{T} \left( \bar{b}_{T, z_1, z_2} - b \right) \right)$ converges weakly to a centered Gaussian distribution with variance given by
\begin{align}
&\frac{64a^3 b}{Q_1^2} \sum_{k=1}^{\infty} \sum_{n=1}^{\infty} \frac{\lambda_k \lambda_n z_{1,k}^2 z_{1,n}^2}{b(\alpha_k - \alpha_n)^2 + 8a^2 (\alpha_k + \alpha_n)} \label{limiting variance of bbarz1z2} \\
&+ \frac{8a b^2}{Q_1^2 Q_2^2} \sum_{k=1}^{\infty} \sum_{n=1}^{\infty} \frac{\lambda_k \lambda_n}{b(\alpha_k - \alpha_n)^2 + 8a^2 (\alpha_k + \alpha_n)} \left( \left( Q_1 z_{2,k} z_{2,n} \sqrt{\alpha_k} - Q_2 z_{1,k} z_{1,n} \sqrt{\alpha_n} \right)^2 \right. \notag \\
&+ \left. \left( Q_1 z_{2,k} z_{2,n} \sqrt{\alpha_n} - Q_2 z_{1,k} z_{1,n} \sqrt{\alpha_k} \right)^2 \right). \notag
\end{align}
\end{theorem}
\begin{proof}
Set $0 \neq z_1 \in \mathrm{Dom}((-A)^{\frac{1}{2}})$ and $0 \neq z_2 \in L^2(D)$. Using formula \eqref{bbar} for the estimator $\bar{b}_{T, z_1, z_2}$ and Lemmas \ref{ito for z2-lemma} and \ref{ito for z1-lemma}, we obtain
\begin{align}
&\sqrt{T} \left( \bar{b}_{T, z_1, z_2} - b \right) = \label{asymptotics for bbarz1z2} \\
&= \frac{\sqrt{T}}{Q_2 \frac{1}{T} \int_0^T \left\langle X_1^{x_0}(t), z_1 \right\rangle_{\mathrm{Dom}((-A)^{\frac{1}{2}})}^2 \, dt} \left( \frac{Q_1}{T} \int_0^T \left\langle X_2^{x_0}(t), z_2 \right\rangle_{L^2(D)}^2 \, dt \right. \notag \\
&\phantom{=} \left. - \frac{b Q_2}{T} \int_0^T \left\langle X_1^{x_0}(t), z_1 \right\rangle_{\mathrm{Dom}((-A)^{\frac{1}{2}})}^2 \, dt \right) \notag \\
&= \frac{\sqrt{T}}{Q_2 \frac{1}{T} \int_0^T \left\langle X_1^{x_0}(t), z_1 \right\rangle_{\mathrm{Dom}((-A)^{\frac{1}{2}})}^2 \, dt} \left( - \frac{Q_1}{4aT} \left( \left\langle E X^{x_0}(T), X^{x_0}(T) \right\rangle_V \right. \right. \notag \\
&\phantom{=} \left. - \left\langle E x_0, x_0 \right\rangle_V \right) + \frac{Q_1}{2aT} \int_0^T \left\langle E X^{x_0}(t), \Phi \, dB(t) \right\rangle_V +  \notag \\
&\phantom{=} \left. + \frac{Q_2}{4aT} \left( \left\langle F X^{x_0}(T), X^{x_0}(T) \right\rangle_V - \left\langle F x_0, x_0 \right\rangle_V \right) - \frac{Q_2}{2aT} \int_0^T \left\langle F X^{x_0}(t), \Phi \, dB(t) \right\rangle_V \right) \notag \\
&= - \frac{Q_1}{4a Q_2 \frac{1}{T} \int_0^T \left\langle X_1^{x_0}(t), z_1 \right\rangle_{\mathrm{Dom}((-A)^{\frac{1}{2}})}^2 \, dt} \frac{1}{\sqrt{T}} \left( \left\langle E X^{x_0}(T), X^{x_0}(T) \right\rangle_V - \left\langle E x_0, x_0 \right\rangle_V \right) \notag \\
&\phantom{=} + \frac{1}{4a \frac{1}{T} \int_0^T \left\langle X_1^{x_0}(t), z_1 \right\rangle_{\mathrm{Dom}((-A)^{\frac{1}{2}})}^2 \, dt} \frac{1}{\sqrt{T}} \left( \left\langle F X^{x_0}(T), X^{x_0}(T) \right\rangle_V - \left\langle F x_0, x_0 \right\rangle_V \right) \notag \\
&\phantom{=} + \frac{1}{2a Q_2 \frac{1}{T} \int_0^T \left\langle X_1^{x_0}(t), z_1 \right\rangle_{\mathrm{Dom}((-A)^{\frac{1}{2}})}^2 \, dt} \frac{1}{\sqrt{T}} \int_0^T \left\langle (Q_1 E - Q_2 F) X^{x_0}(t), \Phi \, dB(t) \right\rangle_V. \notag
\end{align}

The first two terms converge to zero in probability as $T \rightarrow \infty$, since
\begin{equation}
\lim_{T \rightarrow \infty} \frac{1}{T} \int_0^T \left\langle X_1^{x_0}(t), z_1 \right\rangle_{\mathrm{Dom}((-A)^{\frac{1}{2}})}^2 \, dt = \frac{Q_1}{4ab}, \quad \mathbb P-a.s.
\end{equation}
by Theorem \ref{abarbbar - general case} and
\begin{align}
&\lim_{T \rightarrow \infty} \frac{1}{\sqrt{T}} \left( \left\langle E X^{x_0}(T), X^{x_0}(T) \right\rangle_V - \left\langle E x_0, x_0 \right\rangle_V \right) = 0, \quad \mathrm{in} \, \, L^1(\Omega), \notag \\
&\lim_{T \rightarrow \infty} \frac{1}{\sqrt{T}} \left( \left\langle F X^{x_0}(T), X^{x_0}(T) \right\rangle_V - \left\langle F x_0, x_0 \right\rangle_V \right) = 0, \quad \mathrm{in} \, \, L^1(\Omega),
\end{align}
by Lemma \ref{convergence to zero}. Define
\begin{align}
w_1(T) &= \frac{1}{\sqrt{T}} \int_0^T \left\langle (Q_1 E - Q_2 F) X^{x_0}(t), \Phi \, dB(t) \right\rangle_V \notag \\
&= \frac{1}{\sqrt{T}} \int_0^T \sum_{n=1}^{\infty} \sqrt{\lambda_n} \left\langle (Q_1 E - Q_2 F) X^{x_0}(t), \left( \begin{array}{c} 0 \\ e_n \end{array} \right) \right\rangle_V \, d\beta_n (t), \notag
\end{align}
where we have used the representation of $V$--valued Brownian motion $B(t)$.

Furthermore, we compute the scalar product in the above series
\begin{align}
&\left\langle (Q_1 E - Q_2 F) X^{x_0}(t), \left( \begin{array}{c} 0 \\ e_n \end{array} \right) \right\rangle_V = \notag \\
&= \left\langle (Q_1 E_3 - Q_2 F_3) X_1^{x_0}(t), e_n \right\rangle_{L^2(D)} + \left\langle (Q_1 E_4 - Q_2 F_4) X_2^{x_0}(t), e_n \right\rangle_{L^2(D)}. \notag
\end{align}

By the definition of the operators $E_3$, $F_3$, $E_4$, $F_4$, we have
\begin{align}
&\left\langle (Q_1 E_3 - Q_2 F_3) X_1^{x_0}(t), e_n \right\rangle_{L^2(D)} = \notag \\
&= \sum_{k=1}^{\infty} \sum_{l=k+1}^{\infty} \frac{4a}{D_{k,l}} \left( Q_1 z_{2,k} z_{2,l} b \alpha_l (\alpha_k - \alpha_l) - Q_2 z_{1,k} z_{1,l} \sqrt{\alpha_k \alpha_l} \left( 8a^2 + b (\alpha_k - \alpha_l) \right) \right) \times \notag \\
&\phantom{=} \left\langle X_1^{x_0}(t), e_l \right\rangle_{L^2(D)} \delta_{k,n} \notag \\
&\phantom{=} + \sum_{k=1}^{\infty} \sum_{l=k+1}^{\infty} \frac{4a}{D_{k,l}} \left( Q_1 z_{2,k} z_{2,l} b \alpha_k (\alpha_l - \alpha_k) - Q_2 z_{1,k} z_{1,l} \sqrt{\alpha_k \alpha_l} \left( 8a^2 + b (\alpha_l - \alpha_k) \right) \right) \times \notag \\
&\phantom{=} \left\langle X_1^{x_0}(t), e_k \right\rangle_{L^2(D)} \delta_{n,l} \notag \\
&\phantom{=} - \sum_{k=1}^{\infty} 2aQ_2 z_{1,k}^2 \left\langle X_1^{x_0}(t), e_k \right\rangle_{L^2(D)} \delta_{k,n} \notag \\
&= (I) + (II) + (III), \notag
\end{align}
\begin{align}
&\left\langle (Q_1 E_4 - Q_2 F_4) X_2^{x_0}(t), e_n \right\rangle_{L^2(D)} = \notag \\
&= \sum_{k=1}^{\infty} \left( Q_1 z_{2,k}^2 - Q_2 z_{1,k}^2 \right) \left\langle X_2^{x_0}(t), e_k \right\rangle_{L^2(D)} \delta_{k,n} \notag \\
&\phantom{=} + \sum_{k=1}^{\infty} \sum_{l=k+1}^{\infty} \frac{8a^2}{D_{k,l}} \left( Q_1 z_{2,k} z_{2,l} (\alpha_k + \alpha_l) - 2 Q_2 z_{1,k} z_{1,l} \sqrt{\alpha_k \alpha_l} \right) \left\langle X_2^{x_0}(t), e_k \right\rangle_{L^2(D)} \delta_{n,l} \notag \\
&\phantom{=} + \sum_{k=1}^{\infty} \sum_{l=k+1}^{\infty} \frac{8a^2}{D_{k,l}} \left( Q_1 z_{2,k} z_{2,l} (\alpha_k + \alpha_l) - 2 Q_2 z_{1,k} z_{1,l} \sqrt{\alpha_k \alpha_l} \right) \left\langle X_2^{x_0}(t), e_l \right\rangle_{L^2(D)} \delta_{k,n} \notag \\
&= (IV) + (V) + (VI). \notag
\end{align}

By the central limit theorem for martingales, $\mathrm{Law} \left( w_1(T) \right)$ converges weakly to a~Gaussian distribution with a zero mean and variance given by the $\mathbb P-a.s.$ limit
\begin{equation}
\lim_{T \rightarrow \infty} \frac{1}{T} \int_0^T \sum_{n=1}^{\infty} \lambda_n \left( (I) + \ldots + (VI) \right)^2 \, dt.
\end{equation}

The limits of the cross terms are zero due to Lemma \ref{cross terms go to zero} (see the proof of Theorem \ref{asymptotic normality of abar}). We compute the limits of the "diagonal" terms in the $\mathbb P-a.s.$ sense.

\begin{align}
(A) &= \lim_{T \rightarrow \infty} \frac{1}{T} \int_0^T \sum_{n=1}^{\infty} \lambda_n (I)^2 \, dt \notag \\
&= \lim_{T \rightarrow \infty} \frac{1}{T} \int_0^T \sum_{n=1}^{\infty} \lambda_n \left( \sum_{l=n+1}^{\infty} \frac{4a}{D_{n,l}} ( \cdots ) \left\langle X_1^{x_0}(t), e_l \right\rangle_{L^2(D)} \right)^2 \, dt \notag \\
&\stackrel{(*)}{=} \lim_{T \rightarrow \infty} \frac{1}{T} \int_0^T \sum_{n=1}^{\infty} \sum_{l=n+1}^{\infty} \lambda_n \frac{16a^2}{D_{n,l}^2} ( \cdots )^2 \left\langle X_1^{x_0}(t), e_l \right\rangle_{L^2(D)}^2 \, dt \notag \\
&= \sum_{n=1}^{\infty} \sum_{l=n+1}^{\infty} \frac{1}{D_{n,l}^2} \frac{4a \lambda_n \lambda_l}{b \alpha_l} \times \notag \\
&\phantom{=} \left( Q_1 z_{2,n} z_{2,l} b \alpha_l (\alpha_n - \alpha_l) - Q_2 z_{1,n} z_{1,l} \sqrt{\alpha_n \alpha_l} \left( 8a^2 + b (\alpha_n - \alpha_l) \right) \right)^2, \notag
\end{align}
since $\lim_{T \rightarrow \infty} \frac{1}{T} \int_0^T \left\langle X_1^{x_0}(t), e_l \right\rangle_{L^2(D)}^2 \, dt = \frac{\lambda_l}{4ab \alpha_l}$ and in the equality $(*)$ we have also used Lemma \ref{cross terms go to zero} for the cross summands. $( \cdots)$ stands for the bracket from the definition of $(I)$.

Similarly, we have
\begin{align}
(B) &= \lim_{T \rightarrow \infty} \frac{1}{T} \int_0^T \sum_{n=1}^{\infty} \lambda_n (II)^2 \, dt \notag \\
&= \lim_{T \rightarrow \infty} \frac{1}{T} \int_0^T \sum_{n=1}^{\infty} \lambda_n \left( \sum_{k=1}^{n-1} \frac{4a}{D_{k,n}} ( \cdots ) \left\langle X_1^{x_0}(t), e_k \right\rangle_{L^2(D)} \right)^2 \, dt \notag \\
&\stackrel{(*)}{=} \lim_{T \rightarrow \infty} \frac{1}{T} \int_0^T \sum_{n=1}^{\infty} \sum_{k=1}^{n-1} \lambda_n \frac{16a^2}{D_{k,n}^2} ( \cdots )^2 \left\langle X_1^{x_0}(t), e_k \right\rangle_{L^2(D)}^2 \, dt \notag \\
&= \sum_{n=1}^{\infty} \sum_{k=1}^{n-1} \frac{1}{D_{k,n}^2} \frac{4a \lambda_n \lambda_k}{b \alpha_k} \times \notag \\
&\phantom{=} \left( Q_1 z_{2,k} z_{2,n} b \alpha_k (\alpha_n - \alpha_k) - Q_2 z_{1,k} z_{1,n} \sqrt{\alpha_k \alpha_n} \left( 8a^2 + b (\alpha_n - \alpha_k) \right) \right)^2, \notag \\
(C) &= \lim_{T \rightarrow \infty} \frac{1}{T} \int_0^T \sum_{n=1}^{\infty} \lambda_n (III)^2 \, dt \notag \\
&= \lim_{T \rightarrow \infty} \frac{1}{T} \int_0^T \sum_{n=1}^{\infty} 4a^2 Q_2^2 \lambda_n z_{1,n}^4 \left\langle X_1^{x_0}(t), e_n \right\rangle_{L^2(D)}^2 \, dt \notag \\
&= \frac{a Q_2^2}{b} \sum_{n=1}^{\infty} \frac{\lambda_n^2}{\alpha_n} z_{1,n}^4, \notag \\
(D) &= \lim_{T \rightarrow \infty} \frac{1}{T} \int_0^T \sum_{n=1}^{\infty} \lambda_n (IV)^2 \, dt \notag \\
&= \lim_{T \rightarrow \infty} \frac{1}{T} \int_0^T \sum_{n=1}^{\infty} \lambda_n \left( Q_1 z_{2,n}^2 - Q_2 z_{1,n}^2 \right)^2 \left\langle X_2^{x_0}(t), e_n \right\rangle_{L^2(D)}^2 \, dt \notag \\
&= \frac{1}{4a} \sum_{n=1}^{\infty} \lambda_n^2 \left( Q_1 z_{2,n}^2 - Q_2 z_{1,n}^2 \right)^2,
\end{align}
since $\lim_{T \rightarrow \infty} \frac{1}{T} \int_0^T \left\langle X_2^{x_0}(t), e_n \right\rangle_{L^2(D)}^2 \, dt = \frac{\lambda_n}{4a}$.

Next, we have
\begin{align}
(E) &= \lim_{T \rightarrow \infty} \frac{1}{T} \int_0^T \sum_{n=1}^{\infty} \lambda_n (V)^2 \, dt \notag \\
&= \lim_{T \rightarrow \infty} \frac{1}{T} \int_0^T \sum_{n=1}^{\infty} \lambda_n \left( \sum_{k=1}^{n-1} \frac{8a^2}{D_{k,n}} ( \cdots ) \left\langle X_2^{x_0}(t), e_k \right\rangle_{L^2(D)} \right)^2 \, dt \notag \\
&\stackrel{(*)}{=} \lim_{T \rightarrow \infty} \frac{1}{T} \int_0^T \sum_{n=1}^{\infty} \sum_{k=1}^{n-1} \lambda_n \frac{64a^4}{D_{k,n}^2} ( \cdots )^2 \left\langle X_2^{x_0}(t), e_k \right\rangle_{L^2(D)}^2 \, dt \notag \\
&= \sum_{n=1}^{\infty} \sum_{k=1}^{n-1} \frac{1}{D_{k,n}^2} 16a^3 \lambda_n \lambda_k \left( Q_1 z_{2,k} z_{2,n} (\alpha_k + \alpha_n) - 2 Q_2 z_{1,k} z_{1,n} \sqrt{\alpha_k \alpha_n} \right)^2, \notag
\end{align}
\begin{align}
(F) &= \lim_{T \rightarrow \infty} \frac{1}{T} \int_0^T \sum_{n=1}^{\infty} \lambda_n (VI)^2 \, dt \notag \\
&= \lim_{T \rightarrow \infty} \frac{1}{T} \int_0^T \sum_{n=1}^{\infty} \lambda_n \left( \sum_{l=n+1}^{\infty} \frac{8a^2}{D_{n,l}} ( \cdots ) \left\langle X_2^{x_0}(t), e_l \right\rangle_{L^2(D)} \right)^2 \, dt \notag \\
&\stackrel{(*)}{=} \lim_{T \rightarrow \infty} \frac{1}{T} \int_0^T \sum_{n=1}^{\infty} \sum_{l=n+1}^{\infty} \lambda_n \frac{64a^4}{D_{n,l}^2} ( \cdots )^2 \left\langle X_2^{x_0}(t), e_l \right\rangle_{L^2(D)}^2 \, dt \notag \\
&= \sum_{n=1}^{\infty} \sum_{l=n+1}^{\infty} \frac{1}{D_{n,l}^2} 16a^3 \lambda_n \lambda_l \left( Q_1 z_{2,n} z_{2,l} (\alpha_n + \alpha_l) - 2 Q_2 z_{1,n} z_{1,l} \sqrt{\alpha_n \alpha_l} \right)^2. \notag
\end{align}

The limiting variance of $w_1(T)$ is the sum of the six above terms $(A) + \ldots + (F)$, which can be simplified (analogously as in the proof of Theorem \ref{asymptotic normality of abar}) to
\begin{align}
&\frac{16a^3 b Q_2^2}{b} \sum_{k=1}^{\infty} \sum_{n=1}^{\infty} \frac{\lambda_k \lambda_n z_{1,k}^2 z_{1,n}^2}{D_{k,n}} \notag \\
&+ 2a \sum_{k=1}^{\infty} \sum_{n=1}^{\infty} \frac{\lambda_k \lambda_n}{D_{k,n}} \left( \left( Q_1 z_{2,k} z_{2,n} \sqrt{\alpha_k} - Q_2 z_{1,k} z_{1,n} \sqrt{\alpha_n} \right)^2 \right. \notag \\
& \left. + \left( Q_1 z_{2,k} z_{2,n} \sqrt{\alpha_n} - Q_2 z_{1,k} z_{1,n} \sqrt{\alpha_k} \right)^2 \right).
\end{align}

Since the multiplicative factor $\frac{1}{2aQ_2 \frac{1}{T} \int_0^T \left\langle X_1^{x_0}(t), z_1 \right\rangle_{\mathrm{Dom}((-A)^{\frac{1}{2}})}^2 \, dt}$ of $w_1(T)$ on the right--hand side of \eqref{asymptotics for bbarz1z2} converges to $\frac{2b}{Q_1 Q_2}$ as $T \rightarrow \infty$, we arrive at \eqref{limiting variance of bbarz1z2}.
\end{proof}

\subsection{Asymptotic normality of the estimator $\bar{b}_{T, z_1, a}$}
Based on the estimation strategy 2., another estimator of the parameter $b$ may be introduced. Using "observation window" $(z_1, 0)^{\top}$, $z_1 \neq 0$, with the parameter $a$ known, we may define
\begin{equation} \label{bbarz1a}
\bar{b}_{T, z_1, a} = \frac{\left\langle Q z_1, z_1 \right\rangle_{\mathrm{Dom}((-A)^{\frac{1}{2}})}}{4a \frac{1}{T} \int_0^T \left\langle X_1^{x_0}(t), z_1 \right\rangle_{\mathrm{Dom}((-A)^{\frac{1}{2}})}^2 \, dt}.
\end{equation}

This estimator $\bar{b}_{T, z_1, a}$ is also strongly consistent as $T \rightarrow \infty$ by Theorem \ref{abarbbar - general case} and we show its asymptotic normality.

\begin{theorem} \label{asymptotic normality of bbar - z1a}
Let $0 \neq z_1 \in \mathrm{Dom}((-A)^{\frac{1}{2}})$, $a > 0$ be arbitrary. The estimator $\bar{b}_{T, z_1, a}$ is asymp\-to\-ti\-cally normal, i.e., $\mathrm{Law} \left( \sqrt{T} \left( \bar{b}_{T, z_1, a} - b \right) \right)$ converges weakly to a~centered Gaussian distribution with variance given by
\begin{align}
&\frac{64a^3 b}{Q_1^2} \sum_{k=1}^{\infty} \sum_{n=1}^{\infty} \frac{\lambda_k \lambda_n z_{1,k}^2 z_{1,n}^2}{b(\alpha_k - \alpha_n)^2 + 8a^2 (\alpha_k + \alpha_n)} \notag \\
&\phantom{=}+ \frac{8a b^2}{Q_1^2} \sum_{k=1}^{\infty} \sum_{n=1}^{\infty} \frac{\lambda_k \lambda_n (\alpha_k + \alpha_n) z_{1,k}^2 z_{1,n}^2}{b(\alpha_k - \alpha_n)^2 + 8a^2 (\alpha_k + \alpha_n)}. \label{limiting variance of bbarz1a}
\end{align}
\end{theorem}
\begin{proof}
Set $0 \neq z_1 \in \mathrm{Dom}((-A)^{\frac{1}{2}})$. Using formula \eqref{bbarz1a} for the estimator $\bar{b}_{T, z_1, a}$ and Lemma \ref{ito for z1-lemma}, we obtain
\begin{align}
&\sqrt{T} \left( \bar{b}_{T, z_1, a} - b \right) = \label{asymptotics for bbarz1a} \\
&= \frac{\sqrt{T}}{4a \frac{1}{T} \int_0^T \left\langle X_1^{x_0}(t), z_1 \right\rangle_{\mathrm{Dom}((-A)^{\frac{1}{2}})}^2 \, dt} \left( Q_1 - \frac{4ab}{T} \int_0^T \left\langle X_1^{x_0}(t), z_1 \right\rangle_{\mathrm{Dom}((-A)^{\frac{1}{2}})}^2 \, dt \right) \notag \\
&= \frac{1}{4a \frac{1}{T} \int_0^T \left\langle X_1^{x_0}(t), z_1 \right\rangle_{\mathrm{Dom}((-A)^{\frac{1}{2}})}^2 \, dt} \frac{1}{\sqrt{T}} \left( \left\langle F X^{x_0}(T), X^{x_0}(T) \right\rangle_V - \left\langle F x_0, x_0 \right\rangle_V \right) \notag \\
&\phantom{=} - \frac{1}{2a \frac{1}{T} \int_0^T \left\langle X_1^{x_0}(t), z_1 \right\rangle_{\mathrm{Dom}((-A)^{\frac{1}{2}})}^2 \, dt} \frac{1}{\sqrt{T}} \int_0^T \left\langle F X^{x_0}(t), \Phi \, dB(t) \right\rangle_V. \notag
\end{align}

The first term converges to zero in probability as $T \rightarrow \infty$ (see the proof of Theorem \ref{asymptotic normality of bbar}) and define
\begin{align}
w_2(T) &= \frac{1}{\sqrt{T}} \int_0^T \left\langle F X^{x_0}(t), \Phi \, dB(t) \right\rangle_V \notag \\
&= \frac{1}{\sqrt{T}} \int_0^T \sum_{n=1}^{\infty} \sqrt{\lambda_n} \left\langle F X^{x_0}(t), \left( \begin{array}{c} 0 \\ e_n \end{array} \right) \right\rangle_V \, d\beta_n (t). \notag
\end{align}

Since the stochastic integral $w_2(T)$ is a special case of the stochastic integral $w_1(T)$ from the proof of Theorem \ref{asymptotic normality of bbar} (use $E = 0$, $Q_2 = 1$ and omit the minus sign), we will handle it much easier than above. The required scalar products equal to
\begin{align}
&\left\langle F_3 X_1^{x_0}(t), e_n \right\rangle_{L^2(D)} = \notag \\
&= \sum_{k=1}^{\infty} \sum_{l=k+1}^{\infty} \frac{z_{1,k} z_{1,l}}{D_{k,l}} 4a \sqrt{\alpha_k \alpha_l} \left( 8a^2 + b (\alpha_k - \alpha_l) \right) \left\langle X_1^{x_0}(t), e_l \right\rangle_{L^2(D)} \delta_{k,n} \notag \\
&\phantom{=} + \sum_{k=1}^{\infty} \sum_{l=k+1}^{\infty} \frac{z_{1,k} z_{1,l}}{D_{k,l}} 4a \sqrt{\alpha_k \alpha_l} \left( 8a^2 + b (\alpha_l - \alpha_k) \right) \left\langle X_1^{x_0}(t), e_k \right\rangle_{L^2(D)} \delta_{n,l} \notag \\
&\phantom{=} + \sum_{k=1}^{\infty} 2a z_{1,k}^2 \left\langle X_1^{x_0}(t), e_k \right\rangle_{L^2(D)} \delta_{k,n} \notag \\
&= (I) + (II) + (III), \notag
\end{align}
\begin{align}
&\left\langle F_4 X_2^{x_0}(t), e_n \right\rangle_{L^2(D)} = \notag \\
&= \sum_{k=1}^{\infty} z_{1,k}^2 \left\langle X_2^{x_0}(t), e_k \right\rangle_{L^2(D)} \delta_{k,n} \notag \\
&\phantom{=} + \sum_{k=1}^{\infty} \sum_{l=k+1}^{\infty} \frac{z_{1,k} z_{1,l}}{D_{k,l}} 16a^2 \sqrt{\alpha_k \alpha_l} \left\langle X_2^{x_0}(t), e_k \right\rangle_{L^2(D)} \delta_{n,l} \notag \\
&\phantom{=} + \sum_{k=1}^{\infty} \sum_{l=k+1}^{\infty} \frac{z_{1,k} z_{1,l}}{D_{k,l}} 16a^2 \sqrt{\alpha_k \alpha_l} \left\langle X_2^{x_0}(t), e_l \right\rangle_{L^2(D)} \delta_{k,n} \notag \\
&= (IV) + (V) + (VI). \notag
\end{align}

The appropriate limits of the "diagonal" terms equal to
\begin{align}
(A) &= \lim_{T \rightarrow \infty} \frac{1}{T} \int_0^T \sum_{n=1}^{\infty} \lambda_n (I)^2 \, dt = \sum_{n=1}^{\infty} \sum_{l=n+1}^{\infty} \frac{z_{1,n}^2 z_{1,l}^2}{D_{n,l}^2} \frac{4a \lambda_n \lambda_l \alpha_n}{b} \left( 8a^2 + b (\alpha_n - \alpha_l) \right)^2, \notag \\
(B) &= \lim_{T \rightarrow \infty} \frac{1}{T} \int_0^T \sum_{n=1}^{\infty} \lambda_n (II)^2 \, dt = \sum_{n=1}^{\infty} \sum_{k=1}^{n-1} \frac{z_{1,k}^2 z_{1,n}^2}{D_{k,n}^2} \frac{4a \lambda_n \lambda_k \alpha_n}{b} \left( 8a^2 + b (\alpha_n - \alpha_k) \right)^2, \notag \\
(C) &= \lim_{T \rightarrow \infty} \frac{1}{T} \int_0^T \sum_{n=1}^{\infty} \lambda_n (III)^2 \, dt = \frac{a}{b} \sum_{n=1}^{\infty} \frac{\lambda_n^2}{\alpha_n} z_{1,n}^4, \notag \\
(D) &= \lim_{T \rightarrow \infty} \frac{1}{T} \int_0^T \sum_{n=1}^{\infty} \lambda_n (IV)^2 \, dt = \frac{1}{4a} \sum_{n=1}^{\infty} \lambda_n^2 z_{1,n}^4, \notag \\
(E) &= \lim_{T \rightarrow \infty} \frac{1}{T} \int_0^T \sum_{n=1}^{\infty} \lambda_n (V)^2 \, dt = \sum_{n=1}^{\infty} \sum_{k=1}^{n-1} \frac{z_{1,k}^2 z_{1,n}^2}{D_{k,n}^2} 64a^3 \lambda_n \lambda_k \alpha_n \alpha_k, \notag \\
(F) &= \lim_{T \rightarrow \infty} \frac{1}{T} \int_0^T \sum_{n=1}^{\infty} \lambda_n (VI)^2 \, dt = \sum_{n=1}^{\infty} \sum_{l=n+1}^{\infty} \frac{z_{1,n}^2 z_{1,l}^2}{D_{n,l}^2} 64a^3 \lambda_n \lambda_l \alpha_n \alpha_l. \notag
\end{align}

Analogously as above, $\mathrm{Law} \left( w_2(T) \right)$ converges weakly to a~Gaussian distribution with a zero mean and variance given by $(A) + \ldots + (F)$, which can be simplified to
\begin{equation}
\frac{16a^3}{b} \sum_{k=1}^{\infty} \sum_{n=1}^{\infty} \frac{\lambda_k \lambda_n z_{1,k}^2 z_{1,n}^2}{D_{k,n}} + 2a \sum_{k=1}^{\infty} \sum_{n=1}^{\infty} \frac{\lambda_k \lambda_n (\alpha_k + \alpha_n) z_{1,k}^2 z_{1,n}^2}{D_{k,n}}.
\end{equation}

Since the multiplicative factor of $w_2(T)$ on the right--hand side of \eqref{asymptotics for bbarz1a} converges to $ - \frac{2b}{Q_1}$, we arrive at \eqref{limiting variance of bbarz1a}.
\end{proof}

\begin{remark} \label{remark: coordinates}
If we consider some special cases of "observation windows", the formulae for the limiting variances from Theorems \ref{asymptotic normality of abar}, \ref{asymptotic normality of bbar} and \ref{asymptotic normality of bbar - z1a} may be con\-si\-de\-rab\-ly simplified.
\begin{itemize}
\item The estimator $\bar{a}_{T,k}$ (defined by \eqref{abark}) has limiting variance $a$ for any $k \in \mathbb N$, i.e.,
\begin{equation}
\mathrm{Law} \left( \sqrt{T} \left( \bar{a}_{T, k} - a \right) \right) \stackrel{w^*}{\longrightarrow} N(0,a), \quad T \rightarrow \infty.
\end{equation}
Since the result does not depend on $k$ it does not matter which coordinate of the second component we observe. All the estimators $\bar{a}_{T,k}$ behave similarly.

\item The estimator $\bar{b}_{T,j,k}$ (defined by \eqref{bbarjk}) satisfies
\begin{equation}
\mathrm{Law} \left( \sqrt{T} \left( \bar{b}_{T, j, k} - b \right) \right) \stackrel{w^*}{\longrightarrow} N \left(0, \frac{4ab}{\alpha_j} + \frac{2b^2}{a} \right), \quad T \rightarrow \infty,
\end{equation}
for any $j,k \in \mathbb N$, $j \neq k$.

\item However, if $j=k$ then the estimator $\bar{b}_{T,j,j}$ satisfies
\begin{equation}
\mathrm{Law} \left( \sqrt{T} \left( \bar{b}_{T, j, j} - b \right) \right) \stackrel{w^*}{\longrightarrow} N \left(0, \frac{4ab}{\alpha_j} \right), \quad T \rightarrow \infty.
\end{equation}

\item The estimator $\bar{b}_{T, f_j, a}$ (defined by \eqref{bbarz1a}) satisfies
\begin{equation}
\mathrm{Law} \left( \sqrt{T} \left( \bar{b}_{T, f_j, a} - b \right) \right) \stackrel{w^*}{\longrightarrow} N \left(0, \frac{4ab}{\alpha_j} + \frac{b^2}{a} \right), \quad T \rightarrow \infty,
\end{equation}
for any $j \in \mathbb N$.
\end{itemize}

By comparing the last three points, we obtain two following observations, which are illustrated by simulations in Section \ref{section: implementation}: \\

1. Since the limiting variance of the estimator $\bar{b}_{T,j,k}$ is greater than the limiting variance of the estimator $\bar{b}_{T, f_j, a}$, it seems that it is better to know the true value of the parameter $a$ exactly, instead of estimating it. However, if $j=k$ then the li\-mi\-ting variance of the estimator $\bar{b}_{T,j,j}$ is even smaller. So estimating the parameter $a$ by the "window" $(0, e_j)^{\top}$ and then using the "window" $(f_j, 0)^{\top}$ to estimate the parameter $b$ should be even better than knowing $a$ exactly. \\

2. Since $\alpha_j \rightarrow \infty$ as $j \rightarrow \infty$, the limiting variance $\frac{4ab}{\alpha_j}$ gets smaller with bigger $j$. Hence it is better to use the "observation coordinate" with bigger $j$ rather than using the smaller one.
\end{remark}

\section{Example} \label{section: examples}
\begin{example} \label{wave equation}
Consider the stochastic wave equation with Dirichlet boundary conditions
\begin{align}
\frac{\partial^2 u}{\partial t^2} (t, \xi) &= b\Delta u(t, \xi) - 2a \frac{\partial u}{\partial t}(t, \xi) + \eta (t, \xi), \quad (t, \xi) \in \mathbb R_+ \times D, \label{example 1} \\
u(0, \xi) &= u_1(\xi), \quad \xi \in D, \notag \\
\frac{\partial u}{\partial t} (0, \xi) &= u_2(\xi), \quad \xi \in D, \notag \\
u(t, \xi) &= 0, \quad (t, \xi) \in \mathbb R_+ \times \partial D, \notag
\end{align}
where $D \subset \mathbb R^d$ is a bounded domain with a smooth boundary, $\eta$ is a noise process that is the formal time derivative of a space dependent Brownian motion and $a > 0$, $b > 0$ are unknown parameters.

We rewrite the hyperbolic system \eqref{example 1} as an infinite dimensional stochastic di\-ffer\-en\-tial equation \eqref{linear equation with parameters}
\begin{align}
dX(t) &= \mathcal A X(t) \, dt + \Phi \, dB(t), \notag \\
X(0) &= x_0 = \left( \begin{array}{c}
u_1\\
u_2
\end{array} \right) \notag
\end{align}
for $t \geq 0$, setting $A = \Delta|_{\mathrm{Dom}(A)}$, $\mathrm{Dom}(A) = H^2(D) \cap H^1_0 (D)$, $\text{Dom}(\mathcal A) = \mathrm{Dom}(A) \times \mathrm{Dom}( (-A)^{\frac{1}{2}})$ and
$$
\mathcal A = \left( \begin{array}{cc}
0&I\\
bA&-2aI
\end{array} \right).
$$

The operator $\mathcal A$ generates strongly continuous semigroup in the space $V =$ \\ $\mathrm{Dom}((-A)^{\frac{1}{2}}) \times L^2(D)$. The driving process may take a form $B(t) = (0, \tilde{B}(t))^{\top}$, where $(\tilde{B}(t), t \geq 0)$ is a standard cylindrical Brownian motion on $L^2(D)$. The noise $\eta$ is modelled as the formal derivative $\Phi_1 \, \frac{d\tilde{B}(t)}{dt}$, $\Phi_1 \in \mathcal L_2(L^2(D))$ and $\Phi \in \mathcal L_2(V)$ is given by
$$
\Phi = \left( \begin{array}{cc}
0&0\\
0&\Phi_1
\end{array} \right).
$$

With this setup, all assumptions of Section \ref{preliminaries} are fulfilled, so Theorem \ref{abarbbar - general case} may be used for estimation of parameters. Theorems \ref{asymptotic normality of abar}, \ref{asymptotic normality of bbar} and \ref{asymptotic normality of bbar - z1a}, which show asymptotic normality of these estimators, may be applied as well.
\end{example}

\section{Implementation and statistical evidence} \label{section: implementation}
We have generated a trajectory of the solution to the stochastic differential equation \eqref{example 1} from Example \ref{wave equation} in the program $\mathsf R$ by Euler's method (see \cite{iacus}). The setup of Example \ref{wave equation} is specified as follows:

\begin{itemize}
\item $D = (0,1)$ -- We consider the wave equation for the oscillating rod modelled as a function from the space $L^2((0,1))$.
\item The choice of the orthonormal basis of the space $L^2((0,1))$ is
$$
\{e_n (\xi) = \sqrt{2} \sin (n \pi \xi), n = 1, \ldots, N\},
$$
whose elements satisfy the boundary condition $u(t, 0) = 0 = u(t, 1)$ for any $t > 0$.
\item $N = 10$ -- We have restricted the expansion of the previous basis only to $N = 10$ functions. The accuracy of our results may suffer due to this limitation, nevertheless we will show that our results are sufficiently satisfactory.
\item $T = 1000$ -- The length of the time interval.
\item $\Delta t = 0.0001$ -- The mesh of the partition of the time interval $[0, T]$.
\item The intial functions $u_1$ and $u_2$ have the following form
$$
u_1 (\xi) = \sqrt{2} \sum_{n=1}^{N} \sin (n \pi \xi) = u_2 (\xi).
$$
This means that $\left\langle u_1, e_n \right\rangle_{L^2(D)} = 1 = \left\langle u_2, e_n \right\rangle_{L^2(D)}$ for any $n = 1, \ldots, N$, so the initial conditions are the same in all $N$ dimensions.
\item $a = 1$, $b = 0.2$ -- The values of the parameters that are to be estimated.
\item $- \alpha_n = - n^2 \pi^2$ -- The eigenvalues of the operator $A$. With this setup the operator $A$ is the Laplacian operator $A = \Delta |_{\mathrm{Dom}(A)}$ with $\mathrm{Dom}(A) = H^2((0,1)) \cap H_0^1((0,1))$.
\item $\lambda_n = \frac{1000}{n^2}$ -- The eigenvalues of the operator $Q$. (The eigenvalues of the ope\-ra\-tor $\Phi_1$ equal to $\sqrt{\lambda_n}$ for any $n = 1, \ldots, N$.) The eigenvalues are chosen in the way that the sum $\sum_{n=1}^{\infty} \lambda_n$ is convergent. The multiplication factor is chosen in order to increase the values of the $\lambda_n$. Otherwise the noise would be in "higher" dimensions so small that it would be practically vanishing.
\end{itemize}

We will focus on the estimation of the parameter $a$ using the first and tenth coordinate of the second component, i.e., $\bar{a}_{T, k=1}$ and $\bar{a}_{T, k=10}$ (see \eqref{abark}), and the estimation of the parameter $b$ using the estimators $\bar{a}_{T, k=1}$ and $\bar{a}_{T, k=10}$ of $a$ together with the first and tenth coordinate of the second component, respectively, i.e., $\bar{b}_{T, j=1, k=1}$ and $\bar{b}_{T, j=10, k=10}$ (see \eqref{bbarjk}). We will also study the estimators $\bar{b}_{T, j=1, a=1}$ and $\bar{b}_{T, j=10, a=1}$, which depend only on the "observation window" $(f_1, 0)^{\top}$ (or $(f_{10}, 0)^{\top}$) with the parameter $a$ supposed to be known (see \eqref{bbarz1a}).

Using the generated trajectory, we obtained the following results:
$$
\begin{array}{llll}
\bar{a}_{T, k=1} &= 0.9994, & \bar{a}_{T, k=10} &= 0.9978, \\
\bar{b}_{T, j=1, k=1} &= 0.1902, & \bar{b}_{T, j=10, k=10} &= 0.2001, \\
\bar{b}_{T, j=1, a=1} &= 0.1901, & \bar{b}_{T, j=10, a=1} &= 0.1997.
\end{array}
$$

Time evolution of these estimators is depicted in Figures \ref{figure: abar}, \ref{figure: bbar} and \ref{figure: bbaraknown}.

\begin{figure}[h!]
\centering
\subfigure[The estimator $\bar{a}_{t, k=1}$]{\epsfig{file=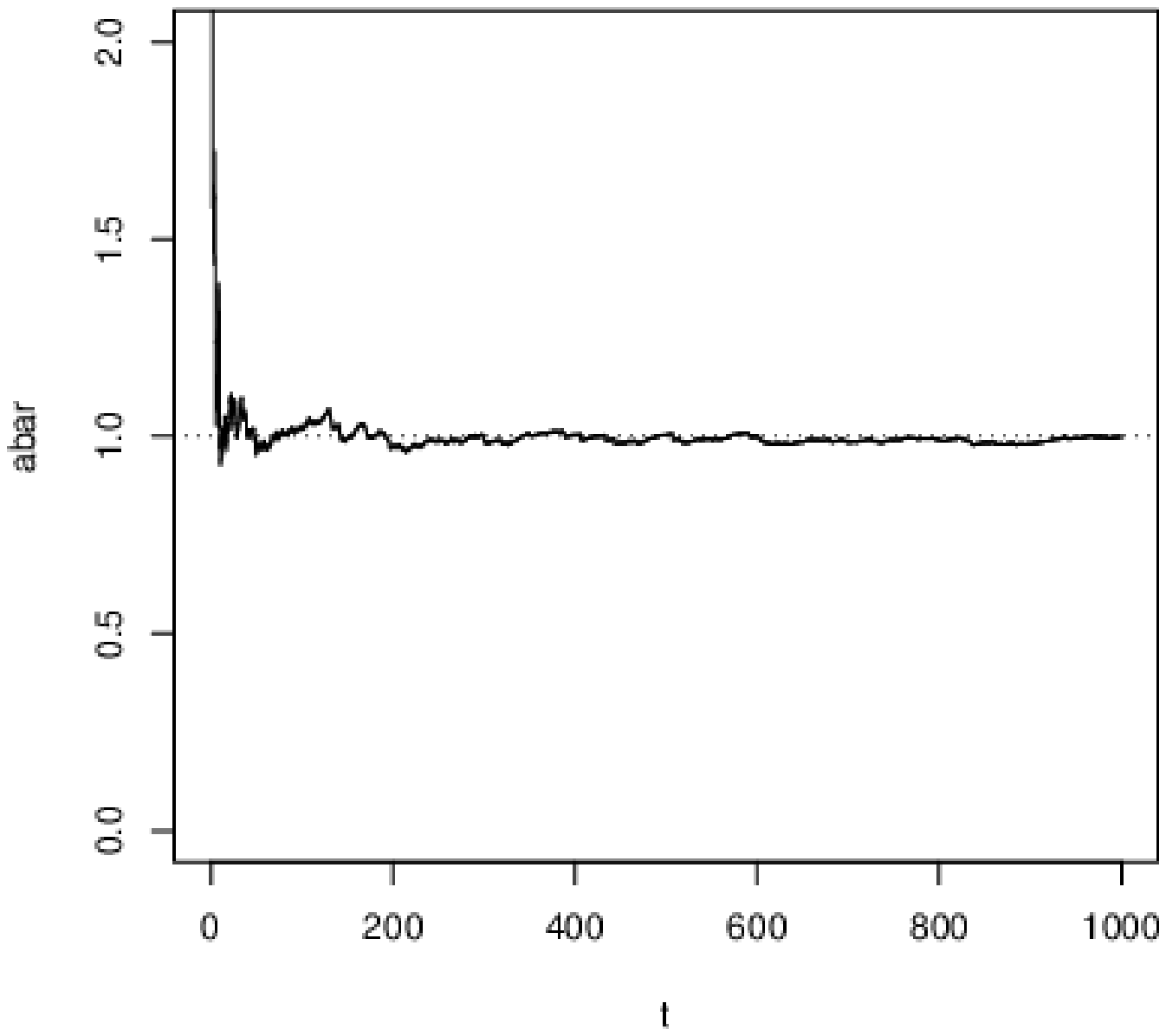,height=0.4\textwidth}}
\hspace{50pt}
\subfigure[The estimator $\bar{a}_{t, k=10}$]{\epsfig{file=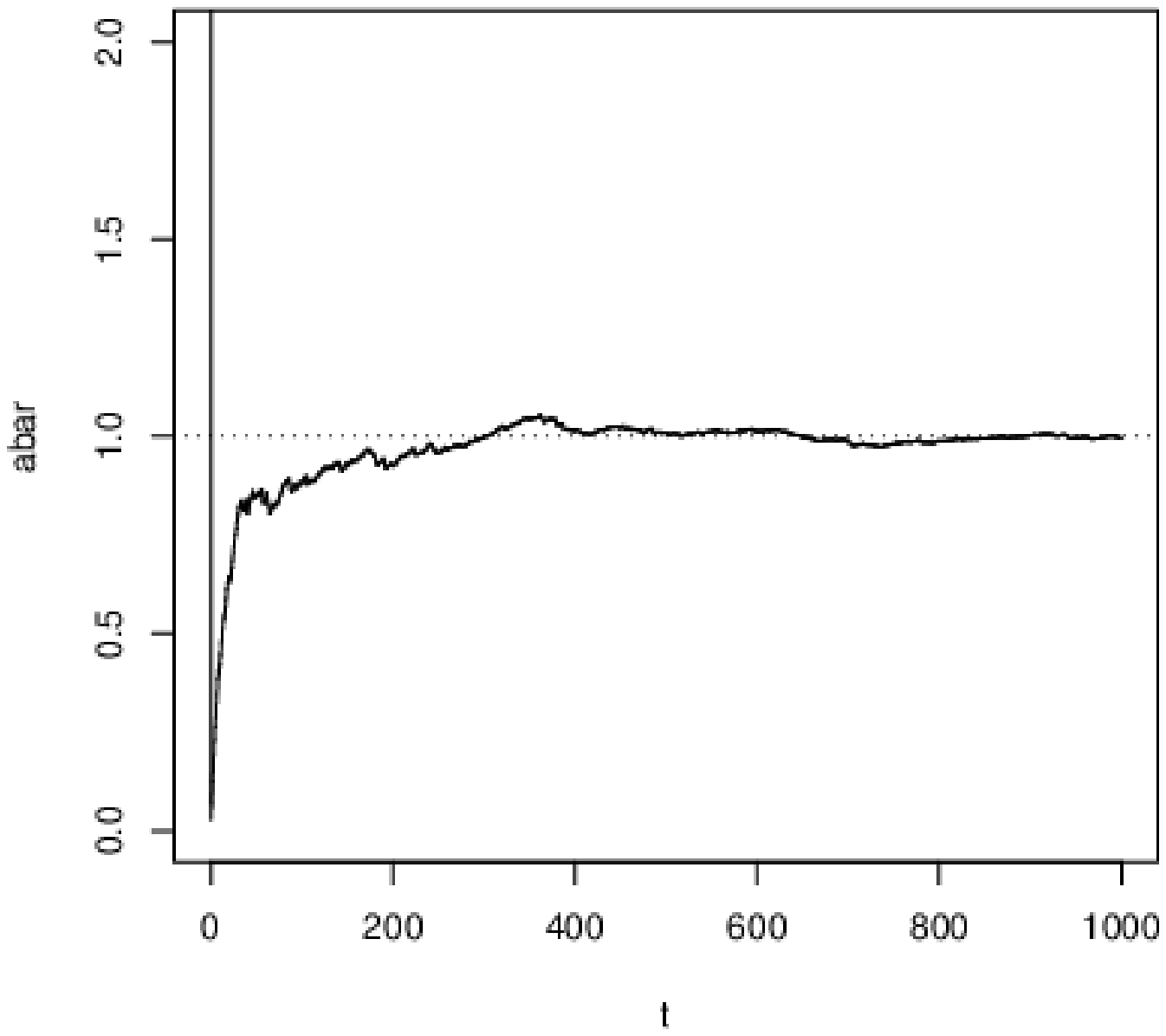,height=0.4\textwidth}}
\caption{The time evolution of the estimators $\bar{a}_{t,k}$ for $k=1$ and $k=10$} \label{figure: abar}
\end{figure}

\begin{figure}[h!]
\centering
\subfigure[The estimator $\bar{b}_{t, j=1, k=1}$]{\epsfig{file=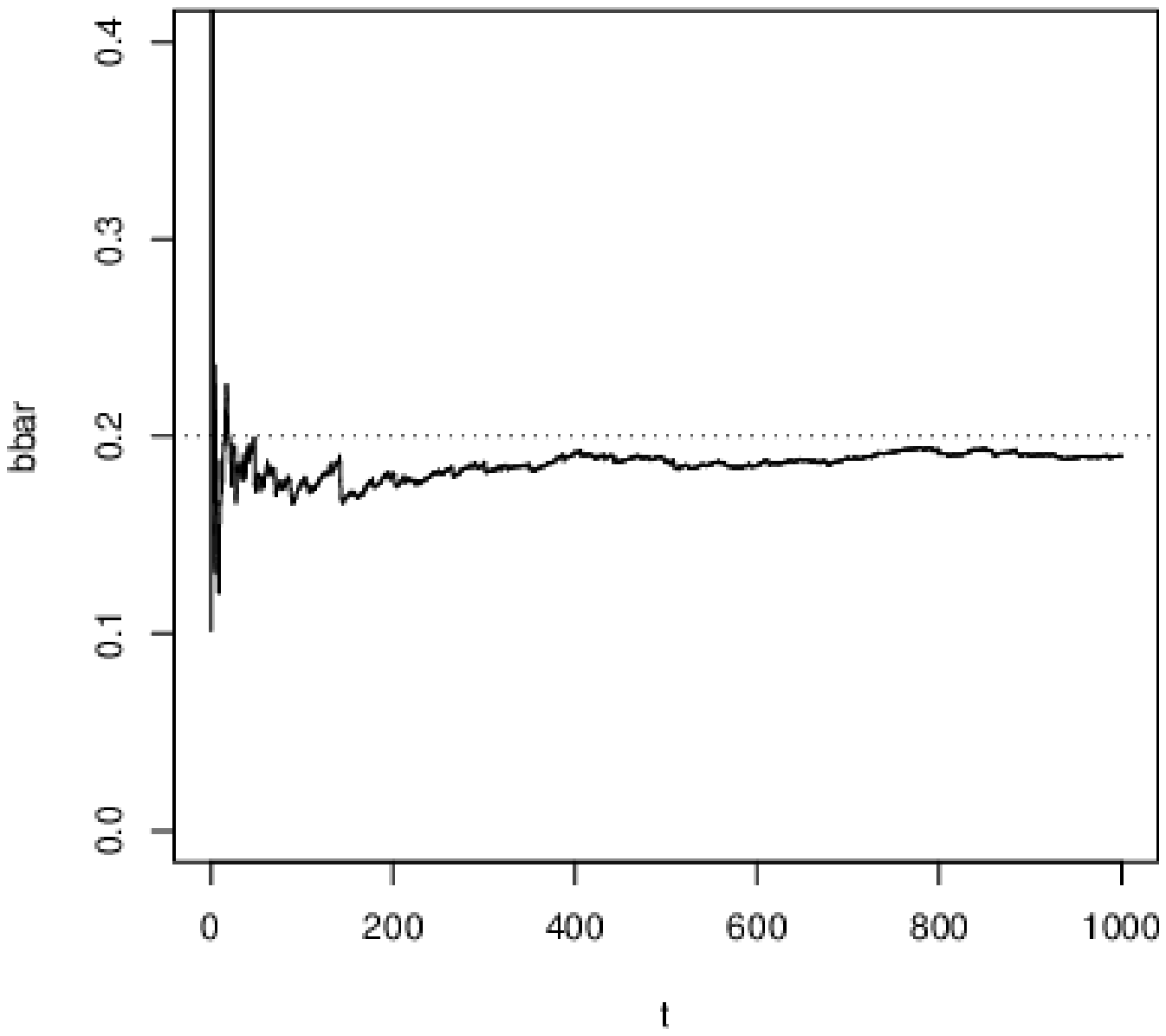,height=0.4\textwidth}}
\hspace{50pt}
\subfigure[The estimator $\bar{b}_{t, j=10, k=10}$]{\epsfig{file=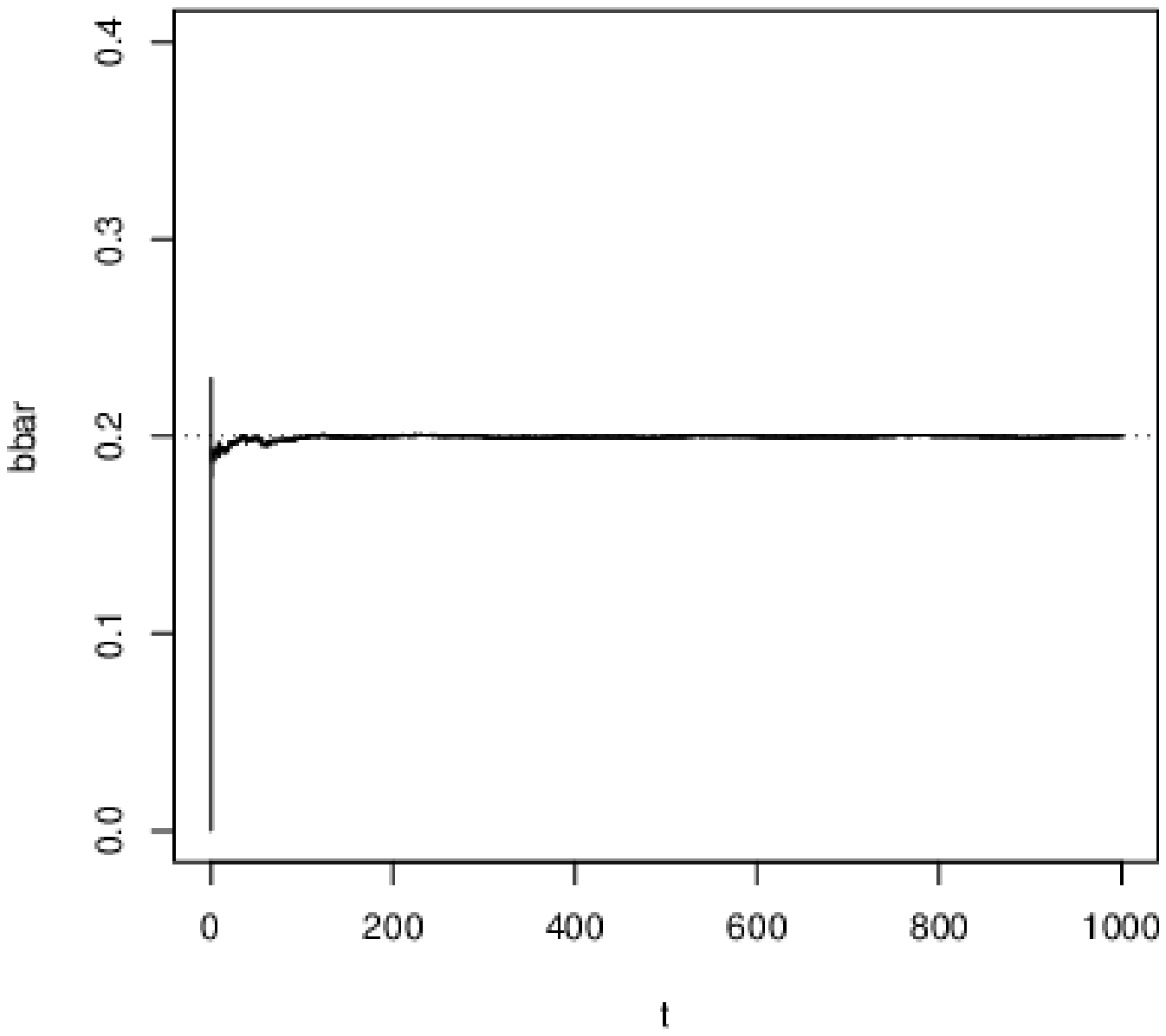,height=0.4\textwidth}}
\caption{The time evolution of the estimators $\bar{b}_{t,j,k}$ for $j=k=1$ and $j=k=10$} \label{figure: bbar}
\end{figure}

\begin{figure}[h!]
\centering
\subfigure[The estimator $\bar{b}_{t, j=1, a=1}$]{\epsfig{file=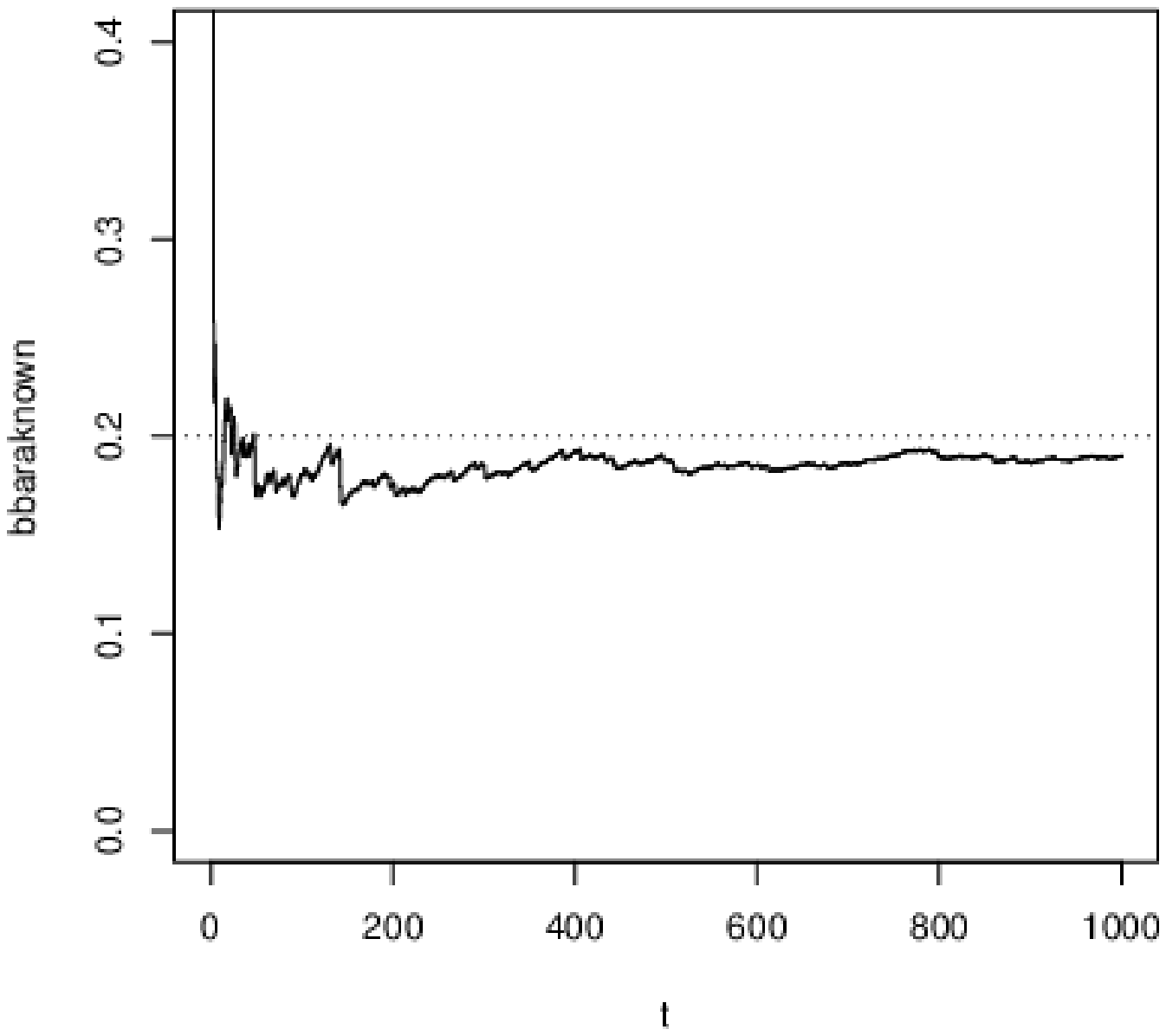,height=0.4\textwidth}}
\hspace{50pt}
\subfigure[The estimator $\bar{b}_{t, j=10, a=1}$]{\epsfig{file=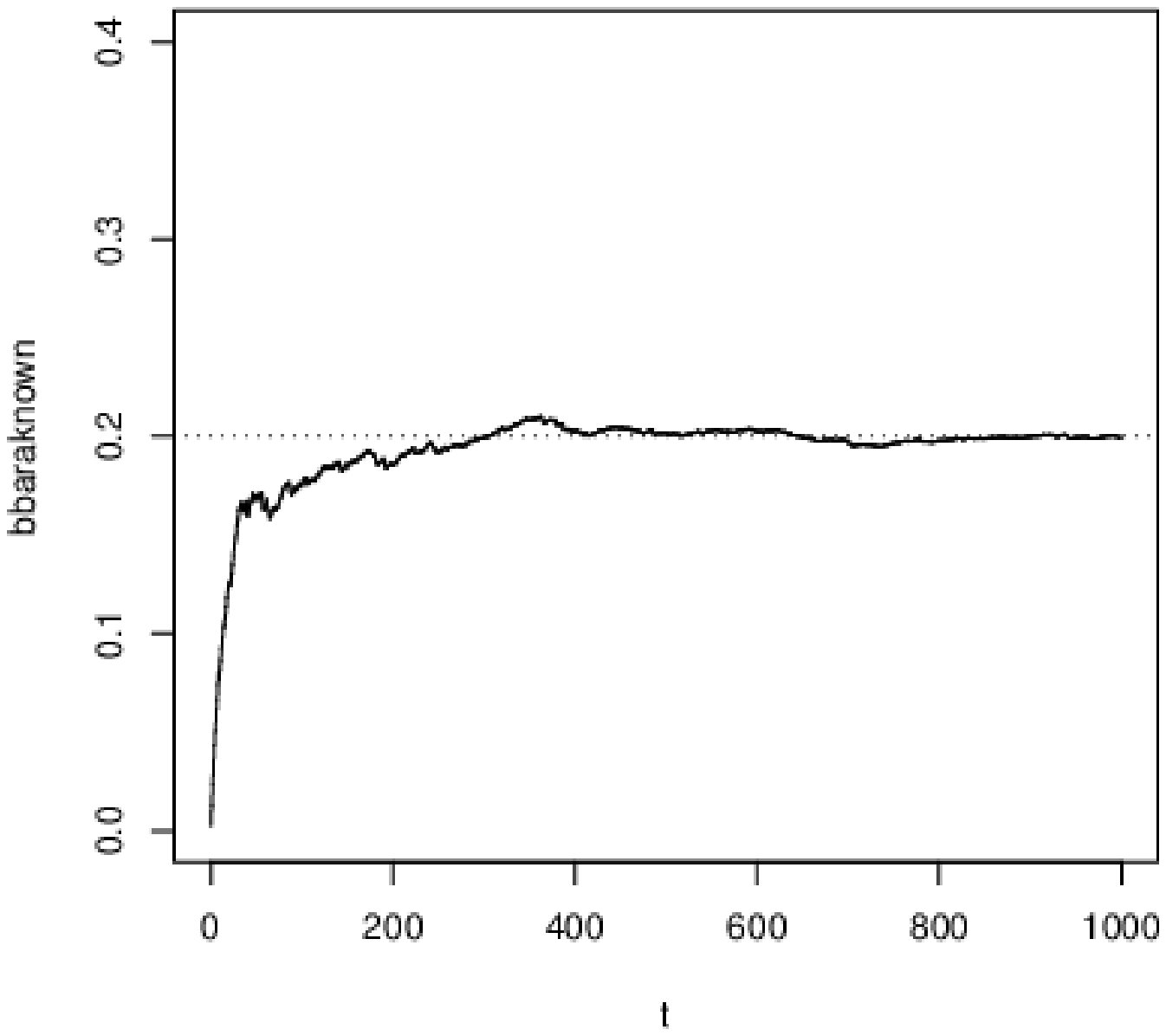,height=0.4\textwidth}}
\caption{The time evolution of the estimators $\bar{b}_{t,j,a}$ for $j=1$, $j=10$ and $a=1$} \label{figure: bbaraknown}
\end{figure}

Although the results seems satisfactory (especially for the estimator $\bar{b}_{T, j=10, k=10}$), we have made $100$ more simulations in a similar manner. The values of the estimators $\bar{a}_{T, k=1}$ and $\bar{a}_{T, k=10}$ are depicted in Figure \ref{figure: abar-overall} and the values of the estimators $\bar{b}_{T, j=1, k=1}$ and $\bar{b}_{T, j=10, k=10}$ are depicted in Figure \ref{figure: bbar-overall}. Moreover, the values of the estimators $\bar{b}_{T, j=1, a=1}$ and $\bar{b}_{T, j=10, a=1}$ are shown in Figure \ref{figure: bbaraknown-overall}. The overall statistics can be found in Tables \ref{Table 1} and \ref{Table 2}.

\begin{figure}[h!]
\centering
\subfigure[The values of $\bar{a}_{T, k=1}$ -- Overall]{\epsfig{file=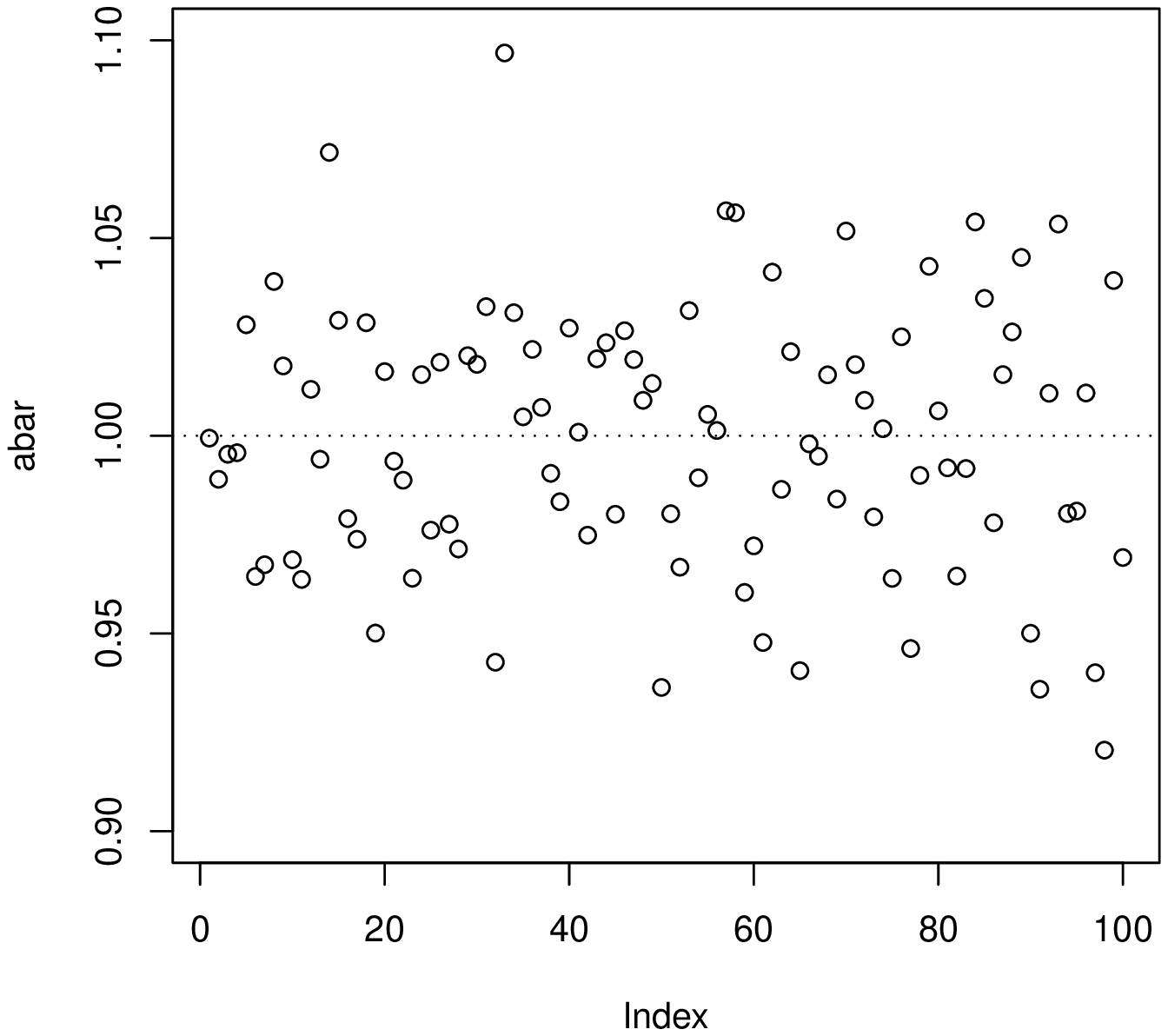,height=0.4\textwidth}}
\hspace{50pt}
\subfigure[The values of $\bar{a}_{T, k=10}$ -- Overall]{\epsfig{file=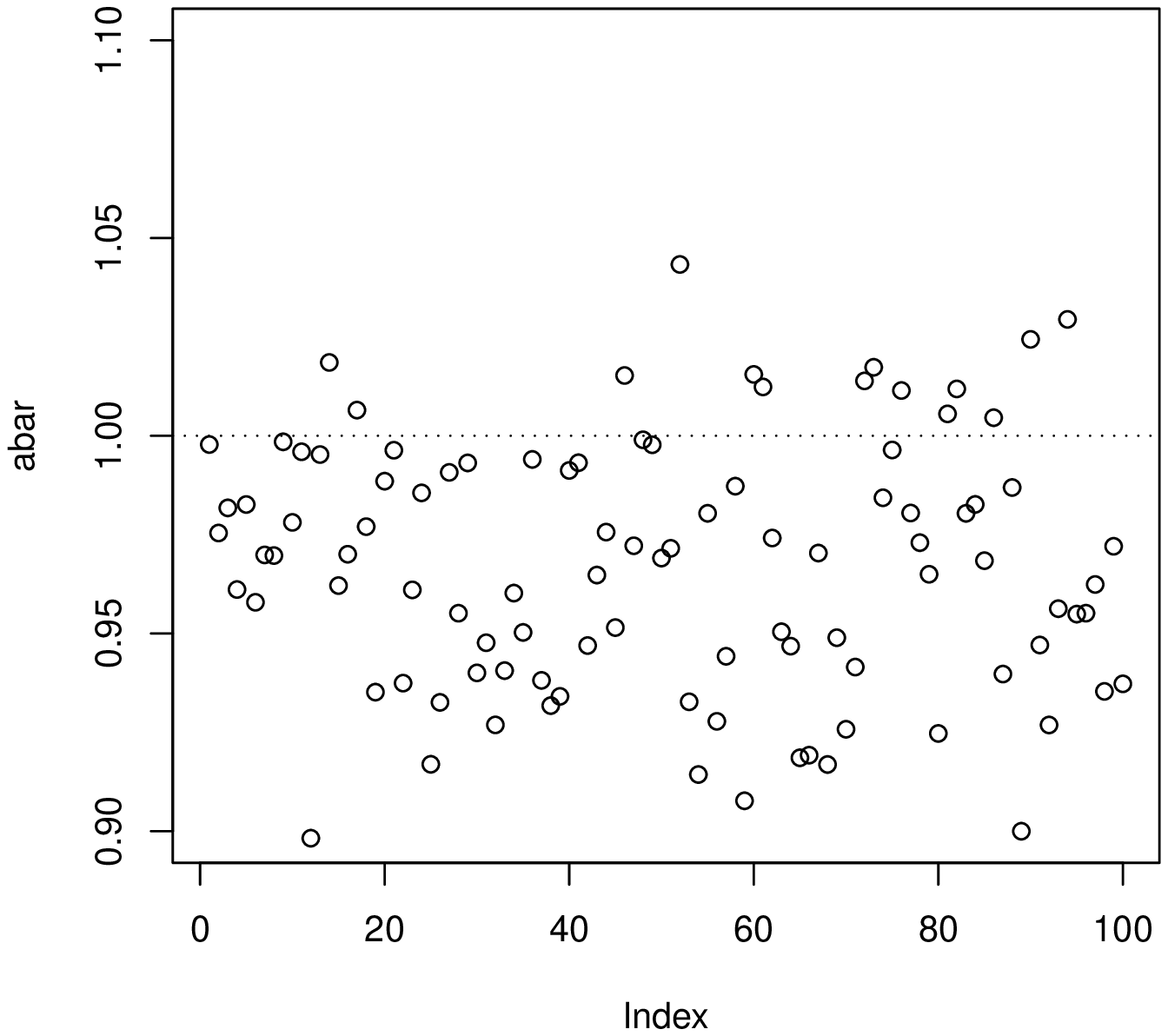,height=0.4\textwidth}}
\caption{The estimators $\bar{a}_{t,k}$ for $k=1$ and $k=10$ based on larger sample} \label{figure: abar-overall}
\end{figure}

\begin{figure}[h!]
\centering
\subfigure[The values of $\bar{b}_{T, j=1, k=1}$ -- Overall]{\epsfig{file=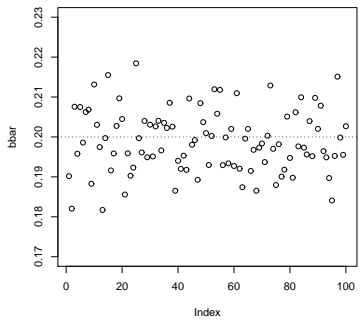,height=0.4\textwidth}}
\hspace{50pt}
\subfigure[The values of $\bar{b}_{T, j=10, k=10}$ -- Overall]{\epsfig{file=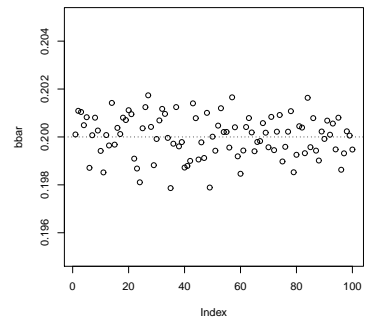,height=0.4\textwidth}}
\caption{The estimators $\bar{b}_{t,j,k}$ for $j=k=1$ and $j=k=10$ based on larger sample} \label{figure: bbar-overall}
\end{figure}

\begin{figure}[h!]
\centering
\subfigure[The values of $\bar{b}_{T, j=1, a=1}$ -- Overall]{\epsfig{file=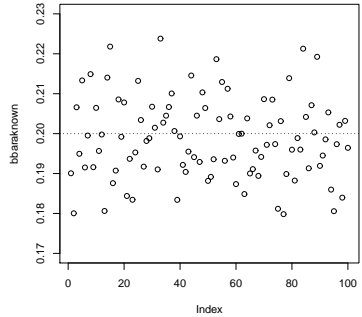,height=0.4\textwidth}}
\hspace{50pt}
\subfigure[The values of $\bar{b}_{T, j=10, a=1}$ -- Overall]{\epsfig{file=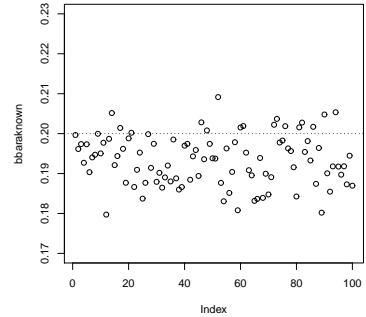,height=0.4\textwidth}}
\caption{The estimators $\bar{b}_{t,j,a}$ for $j=1$, $j=10$ and $a=1$ based on larger sample} \label{figure: bbaraknown-overall}
\end{figure}

\begin{table}[h!]
\centering
\begin{tabular}{|l||c|c|} \hline
& $\bar{a}_{T, k=1}$ & $\bar{a}_{T, k=10}$ \\ \hline \hline
Mean & 0.9995 & 0.9673 \\ \hline
Var & 1.1039 & 1.0036 \\ \hline
Var -- Theoretical & 1.0000 & 1.0000 \\ \hline
Relative error -- Maximal & 10 \% & 10 \% \\ \hline
Relative error -- Typical & $\leq$ 4 \% & $\leq$ 6 \% \\ \hline
$p$--value & 0.9161 & 0.6340 \\ \hline
\end{tabular}
\bigskip
\caption{The results of the simulations -- Part I} \label{Table 1}
\end{table}

\begin{table}[h!]
\centering
\begin{tabular}{|l||c|c|c|c|} \hline
& $\bar{b}_{T, j=1, k=1}$ & $\bar{b}_{T, j=10, k=10}$ & $\bar{b}_{T, j=1, a=1}$ & $\bar{b}_{T, j=10, a=1}$ \\ \hline \hline
Mean & 0.1989 & 0.2000 & 0.1988 & 0.1935 \\ \hline
Var & 0.0616 & 0.0008 & 0.1033 & 0.0402 \\ \hline
Var -- Theoretical & 0.0811 & 0.0008 & 0.1211 & 0.0408 \\ \hline
Relative error -- Maximal & 9 \% & 1 \% & 12 \% & 10 \% \\ \hline
Relative error -- Typical & $\leq$ 5 \% & $\leq$ 0.6 \% & $\leq$ 6 \% & $\leq$ 6 \% \\ \hline
$p$--value & 0.7904 & 0.2986 & 0.2825 & 0.6668 \\ \hline
\end{tabular}
\bigskip
\caption{The results of the simulations -- Part II} \label{Table 2}
\end{table}

The row "Var" stands for the variance of $\sqrt{T}(\bar{a}_{T,k} - a)$ (and its analogues in the following columns). The actual variances of the estimators are $1000$ times smaller. The theoretical values of the limiting variances (see formulae in Remark \ref{remark: coordinates}) can be found in the row "Var -- Theoretical".

Since the absolute errors of the estimators can be viewed in Figures \ref{figure: abar-overall}, \ref{figure: bbar-overall} and \ref{figure: bbaraknown-overall}, we mention only relative errors: maximal (which is the relative error of the worst estimator) and typical (that is the level below which $75$ \% of the errors belong).

The $p$--values of the Wilk--Shapiro test of normality can be found in the last row. Since they are greater than $0.05$, we do not reject the hypothesis of normality on $5 \%$--significance level. The Q--Q plots of the centered and rescaled estimators are shown in Figures \ref{figure: AN for abar}, \ref{figure: AN for bbar} and \ref{figure: AN for bbaraknown}.

\begin{figure}[h!]
\centering
\subfigure[Q--Q plot of $\sqrt{T}(\bar{a}_{T, k=1} - a)$]{\epsfig{file=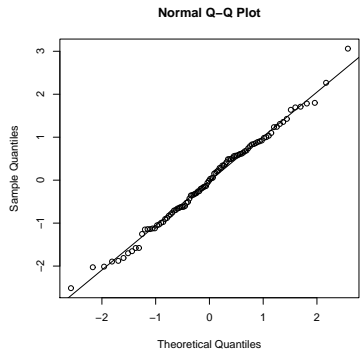,height=0.4\textwidth}}
\hspace{50pt}
\subfigure[Q--Q plot of $\sqrt{T}(\bar{a}_{T, k=10} - a)$]{\epsfig{file=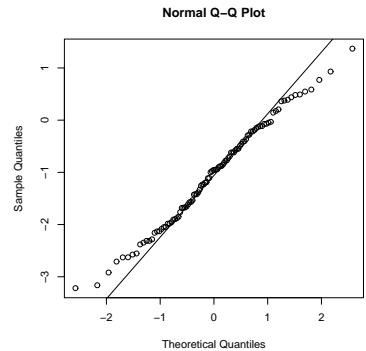,height=0.4\textwidth}}
\caption{Asymptotic normality of $\bar{a}_{T, k}$ for $k=1$ and $k=10$} \label{figure: AN for abar}
\end{figure}

\begin{figure}[h!]
\centering
\subfigure[Q--Q plot of $\sqrt{T}(\bar{b}_{T, j=1, k=1} - b)$]{\epsfig{file=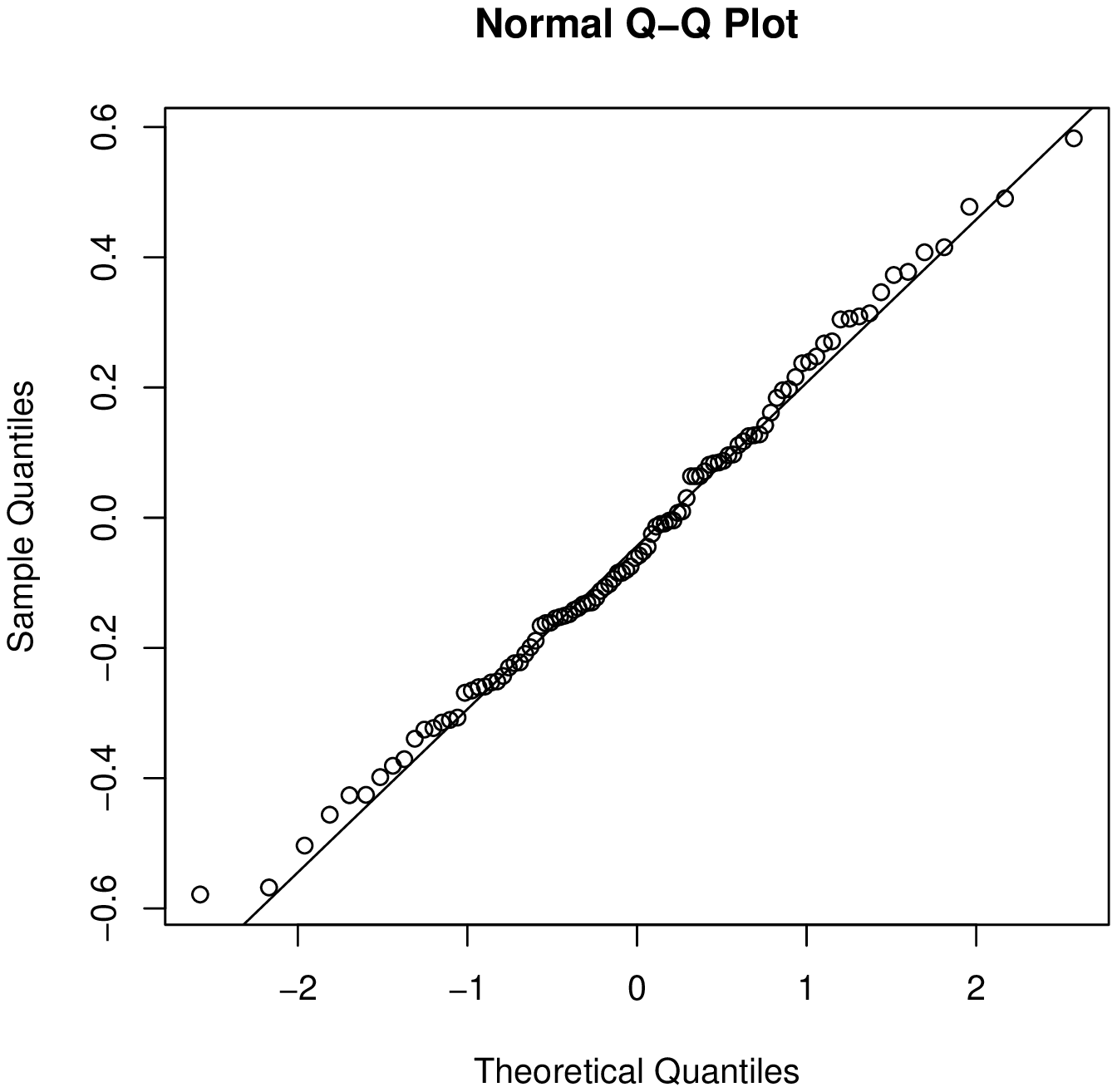,height=0.4\textwidth}}
\hspace{50pt}
\subfigure[Q--Q plot of $\sqrt{T}(\bar{b}_{T, j=10, k=10} - b)$]{\epsfig{file=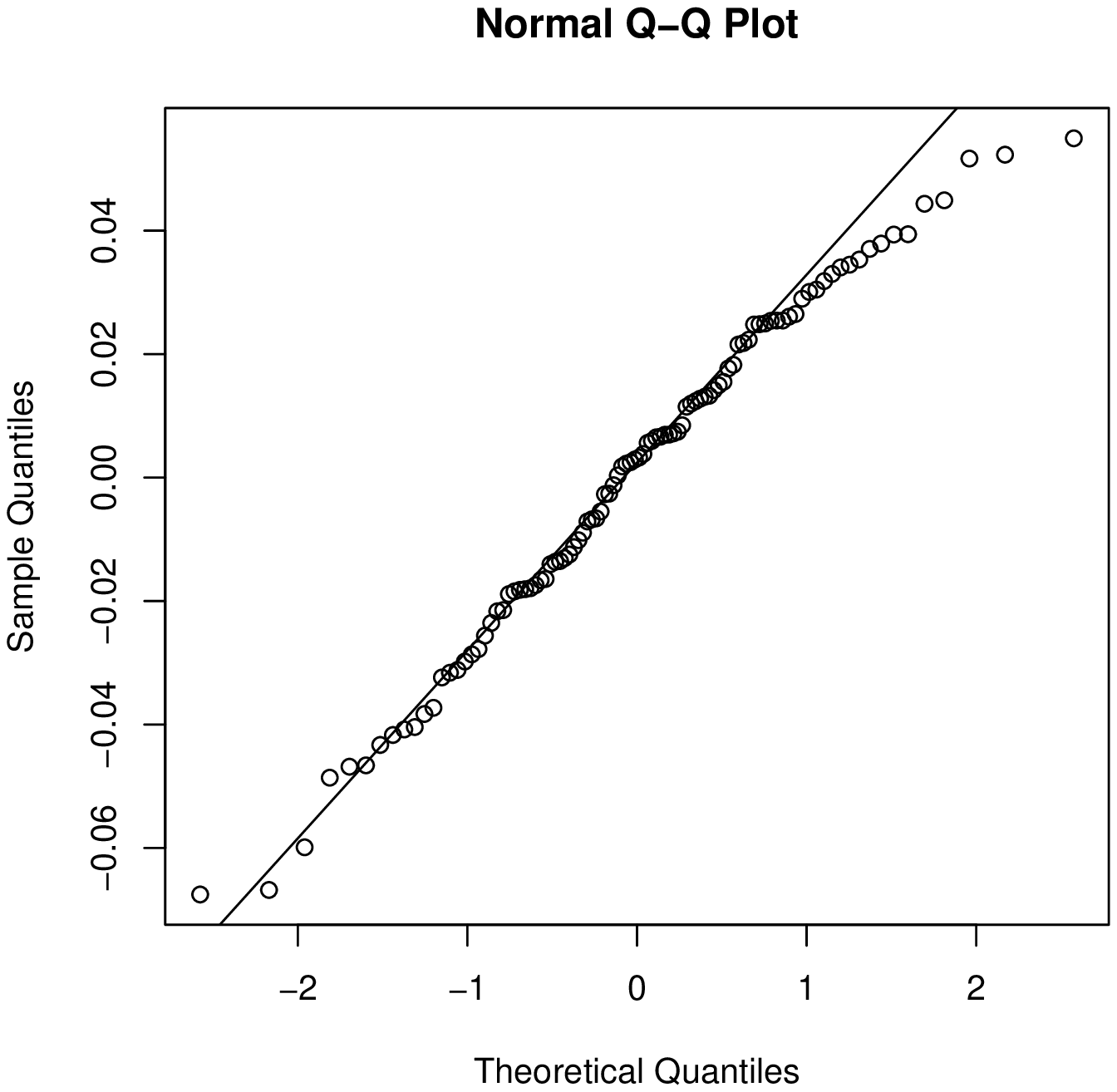,height=0.4\textwidth}}
\caption{Asymptotic normality of $\bar{b}_{T, j, k}$ for $j=k=1$ and $j=k=10$} \label{figure: AN for bbar}
\end{figure}

\begin{figure}[h!]
\centering
\subfigure[Q--Q plot of $\sqrt{T}(\bar{b}_{T, j=1, a=1} - b)$]{\epsfig{file=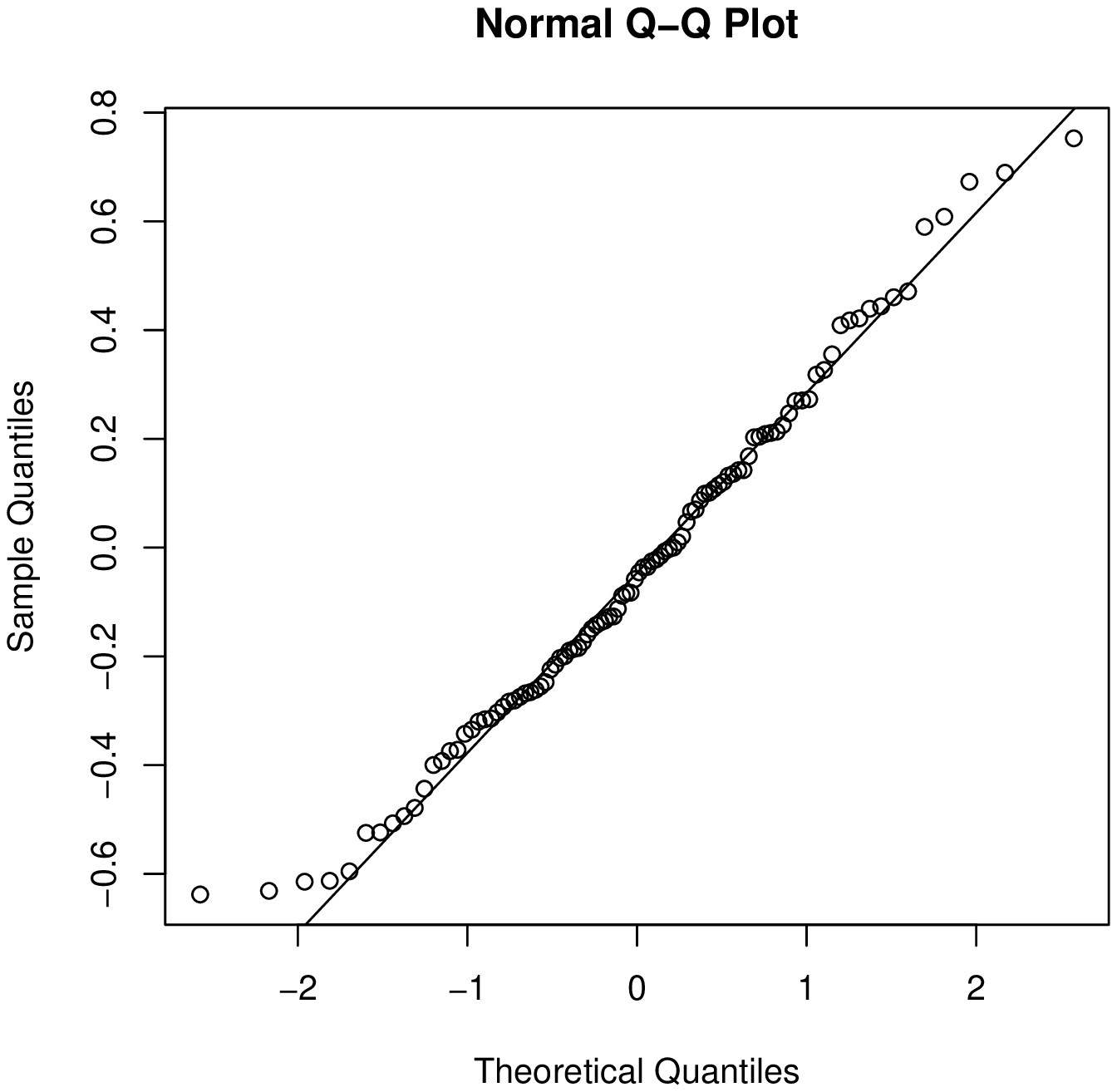,height=0.4\textwidth}}
\hspace{50pt}
\subfigure[Q--Q plot of $\sqrt{T}(\bar{b}_{T, j=10, a=1} - b)$]{\epsfig{file=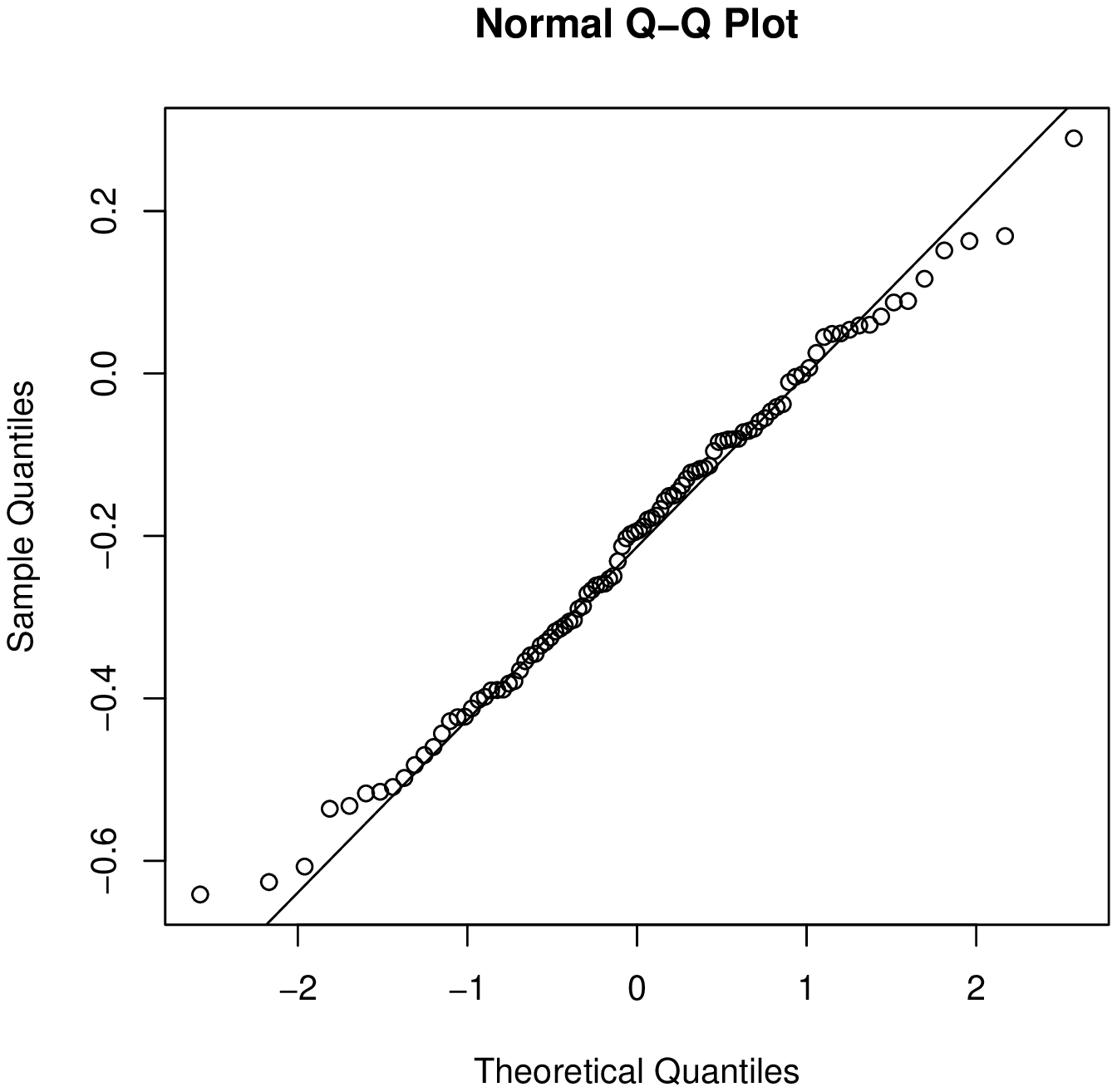,height=0.4\textwidth}}
\caption{Asymptotic normality of $\bar{b}_{T, j, a}$ for $j=1$, $j=10$ and $a=1$} \label{figure: AN for bbaraknown}
\end{figure}

From the previous simulations the main three observations follow:

\begin{itemize}
\item The estimators $\bar{a}_{T, k=1}$ and $\bar{a}_{T, k=10}$ behave similarly (the estimator $\bar{a}_{T, k=10}$ would require some bigger time $T$, though), however there is a big difference between the estimators of the parameter $b$. Not only that the estimator $\bar{b}_{T, j=10, k=10}$ behaves better than the estimator $\bar{b}_{T, j=1, k=1}$ (it has better mean and lesser variance and relative errors), but also by comparing the estimator $\bar{b}_{T, j=1, k=1}$ with $\bar{b}_{T, j=1, a=1}$ (and $\bar{b}_{T, j=10, k=10}$ with $\bar{b}_{T, j=10, a=1}$), it seems that it is better to work with the paramater $a$ unknown. (See Remark \ref{remark: coordinates}.)
\item From the comparing of the rows "Var" and "Var--Theoretical" it seems that the computed limiting variances from Remark \ref{remark: coordinates} are accurate.
\item From the Figures \ref{figure: AN for abar}, \ref{figure: AN for bbar} and \ref{figure: AN for bbaraknown} and from the results of Wilk--Shapiro tests it seems that the estimators are asymptotically normal as prescribed by Theorems \ref{asymptotic normality of abar}, \ref{asymptotic normality of bbar} and \ref{asymptotic normality of bbar - z1a}.
\end{itemize}

These simulations precisely match the results obtained in the theoretical part of the paper.

\end{document}